\documentclass[11pt]{article}
\usepackage[margin=1in]{geometry}
\usepackage[utf8]{inputenc}
\usepackage[T1]{fontenc}
\usepackage[english]{babel}
\usepackage{lmodern}
\usepackage{graphicx}
\usepackage{wrapfig}
\usepackage[subrefformat=parens,labelformat=parens]{subfig}
\usepackage{caption}
\usepackage{paralist}
\usepackage{enumitem}  
\usepackage{amsmath}

\usepackage[linesnumbered,ruled,vlined]{algorithm2e}

\graphicspath{{./figures/}}

\makeatletter
\renewcommand{\p@enumii}{}
\makeatother

\newcommand{\titel}{Minors of non-hamiltonian polyhedra and the Herschel family}
\usepackage[usenames]{color}
\definecolor{hellblau}{rgb}{0.2,0.4,1} 
\definecolor{dunkelblau}{rgb}{0,0,0.8}
\definecolor{dunkelgruen}{rgb}{0,0.5,0}
\usepackage[
pdftex,
colorlinks,
linkcolor=dunkelblau,
urlcolor=dunkelblau,
citecolor=dunkelgruen,
bookmarks=true,
linktocpage=true,
pdfauthor={On-Hei Solomon Lo},
pdfsubject={}
]{hyperref}
\usepackage{pdfpages}
\usepackage{amsmath}
\usepackage{amsthm}
\allowdisplaybreaks
\usepackage{amsfonts}

\theoremstyle{plain}
\newtheorem{satz}{Satz}[section] 
\newtheorem{theorem}[satz]{Theorem}
\newtheorem{lemma}[satz]{Lemma}

\newtheorem{corollary}[satz]{Corollary}
\newtheorem{question}[satz]{Question}
\newtheorem{claim}[satz]{Claim}

\theoremstyle{remark}

\theoremstyle{definition}

\newtheorem{conjecture}[satz]{Conjecture}

\begin{document}
	\title{\titel}
	\author{
		On-Hei Solomon Lo\thanks {School of Mathematical Sciences, Key Laboratory of Intelligent Computing and Applications (Tongji University), Ministry of Education, Tongji University, Shanghai 200092, China} \and
		Kenta Ozeki\thanks {Faculty of Environment and Information Sciences, Yokohama National University, 79-2 Tokiwadai, Hodogaya-ku, Yokohama 240-8501, Japan}\\
	}
	\date{}
	\maketitle
	
	\begin{abstract}
		We show that every non-hamiltonian polyhedron contains the Herschel graph as a minor, implying that the Herschel graph is the unique minor-minimal non-hamiltonian polyhedron. Our approach unifies many previously known results on minors of non-hamiltonian polyhedra, while strengthening them with significantly shorter, non-computer-assisted proofs. As an application, we characterize non-hamiltonian polyhedra with no $K_{2,6}$ minor, resolving a conjecture of Ellingham, Marshall, and Royle.\\
		
	\end{abstract}
	
	\sloppy
	
	\section{Introduction}

	A graph is \emph{polyhedral} if it is the skeleton of a polyhedron. Steinitz's theorem characterizes polyhedral graphs as those that are planar and 3-connected.

In 1857, Sir William Rowan Hamilton introduced the Icosian game, whose goal was to find a cycle in the dodecahedron passing through every vertex exactly once. A cycle with this property in any graph is now called a \emph{Hamilton cycle}, and a graph containing a Hamilton cycle is called \emph{hamiltonian}. The study of Hamilton cycles has since become a central subject in graph theory.

Interestingly, one year earlier, in 1856, during his study of representations of polyhedra, Thomas Penyngton Kirkman~\cite{Kirkman1856} exhibited a polyhedral graph $\mathfrak{K}$, depicted in Figure~\ref{subfig:H13}, that admits no Hamilton cycle; see also~\cite{Biggs1976,Biggs1981}. We call this graph the \emph{Kirkman graph}.

In 1956, Tutte~\cite{Tutte1956} proved that every polyhedral graph that cannot be disconnected by removing three vertices is hamiltonian. However, it is well known that there are infinitely many non-hamiltonian polyhedra. The Kirkman graph $\mathfrak{K}$ appears to be the first known such example, while the smallest is the \emph{Herschel graph} $\mathfrak{H}$, depicted in Figure~\ref{fig:Herschel}. The minimality of the Herschel graph was first noted by Coxeter~\cite{Coxeter1948} in 1948 and later rigorously proven by Barnette and Jucovi\v{c}~\cite{Barnette1970} and Dillencourt~\cite{Dillencourt1996}, who showed that it is the minimal non-hamiltonian polyhedron in terms of both the number of vertices and the number of edges. We note that it is straightforward to see that neither $\mathfrak{H}$ nor $\mathfrak{K}$ contains a Hamilton cycle, since both are bipartite and have an odd number of vertices.

	\begin{figure}[!ht]
		\centering
		\subfloat[A polyhedral representation.\label{subfig:HerschelPolyhedron}]{%
			\includegraphics[scale=.21]{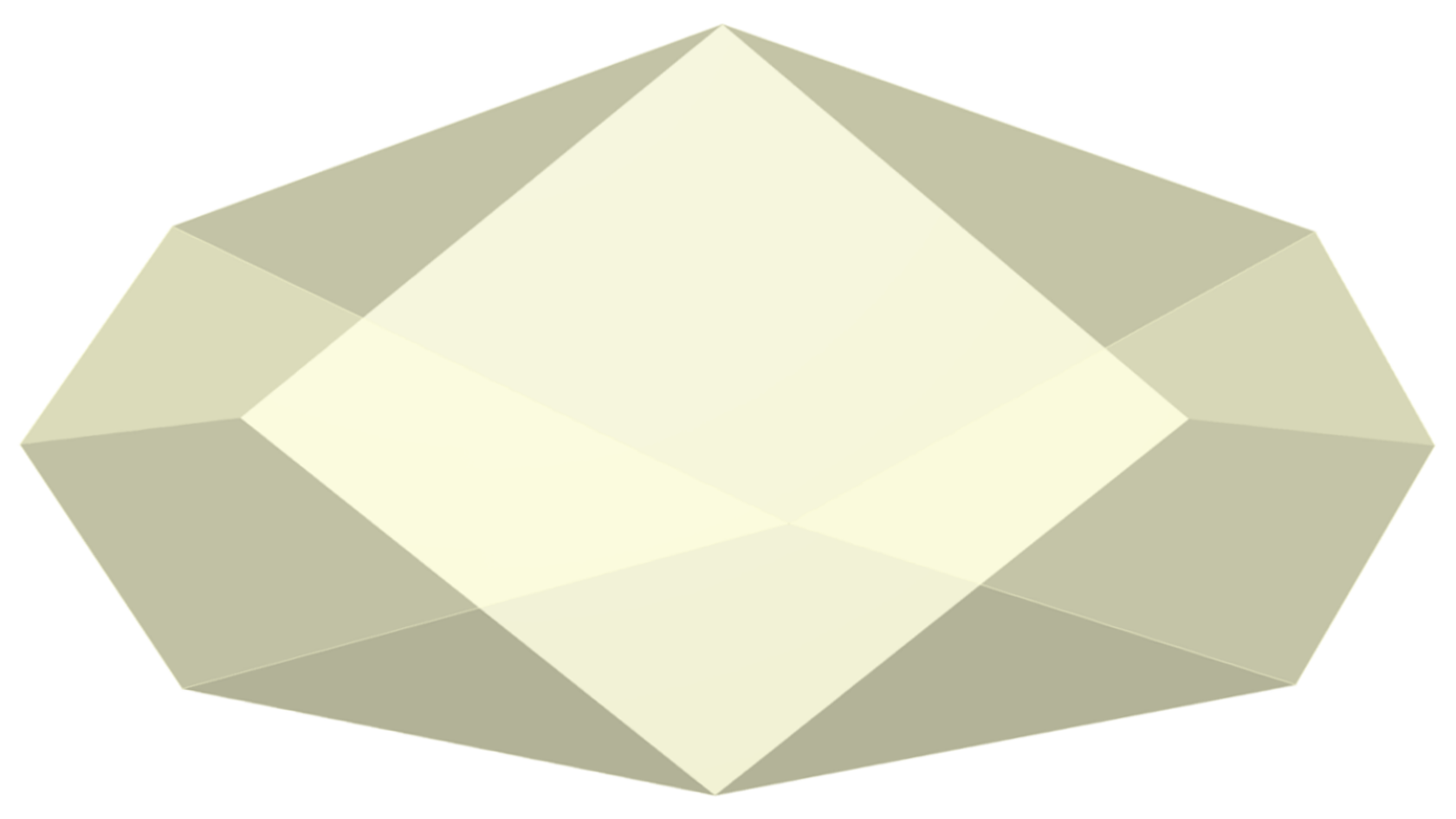}
		}
		\hspace{30pt}
		\subfloat[A graph representation.\label{subfig:Herschel}]{%
			\includegraphics[scale=1]{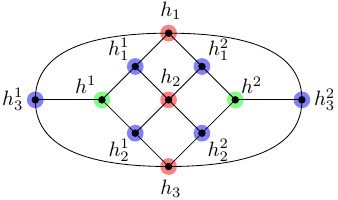}
		}
		\caption{The Herschel graph $\mathfrak{H}$.}
		\label{fig:Herschel}
	\end{figure}
	
	Recall that a graph $H$ is a \emph{minor} of a graph $G$ if $H$ can be obtained from $G$ by a sequence of edge deletions, vertex deletions, and edge contractions. The minor relation defines a partial order within any class of graphs. 
	
	A classical theorem of Wagner~\cite{Wagner1937} and Hall~\cite{Hall1943} states that a 3-connected graph is non-planar if and only if it is isomorphic to \(K_5\) or contains a \(K_{3,3}\) minor. Consequently, Tutte’s theorem is equivalent to the assertion that every 4-connected non-hamiltonian graph contains a \(K_{3,3}\) minor. This naturally raises the following question.
	
	\begin{question}\label{que}
		For a $3$-connected non-hamiltonian graph that contains no \(K_{3,3}\) minor, which other minor is forced to appear? Equivalently, which minor must every non-hamiltonian polyhedral graph contain?
	\end{question} 
	
	Motivated by Question~\ref{que}, Ellingham, Marshall, Ozeki, and Tsuchiya~\cite{Ellingham2019} proved the following result.
	
	\begin{theorem}[Ellingham, Marshall, Ozeki, and Tsuchiya~\cite{Ellingham2019}]\label{thm:K25}
		Every non-hamiltonian polyhedral graph contains a $K_{2,5}$ minor. 
	\end{theorem}
	
	They also observed that there are infinitely many non-hamiltonian polyhedral graphs that do not contain any $K_{2,6}$ minor. The following conjecture was discussed in~\cite{Marshall2014, Ellingham2019, Solava2019} and formulated in~\cite{O'Connell2018}.
	
	\begin{conjecture}[Ellingham, Marshall, and Royle (see~\cite{O'Connell2018, Solava2019})]\label{con:K26}
		Let $G$ be a non-hamiltonian polyhedral graph with at least $16$ vertices. If $G$ has no $K_{2,6}$ minor, then $G$ is isomorphic to the graph depicted in Figure~\ref{subfig:K26f1} or Figure~\ref{subfig:K26f2}, where each dashed edge may be present or absent.
	\end{conjecture}
	
	\begin{figure}[!ht]
		\centering
		\begin{tabular}{cc}
			\subfloat[The graph $\mathfrak{H}_n^\bullet$ on $n \ge 11$ vertices.\label{subfig:K26f1}]{
				\includegraphics[scale=1]{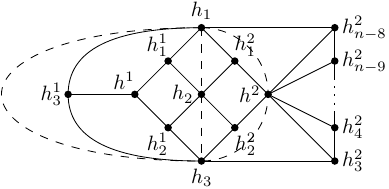}
			} & 
			\subfloat[The graph $\mathfrak{H}_n^\circ$ on $n \ge 13$ vertices.\label{subfig:K26f2}]{
				\includegraphics[scale=1]{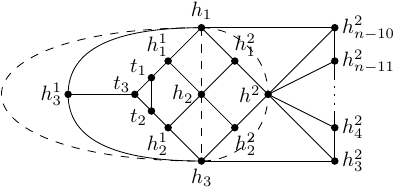}
			} \\
			\subfloat[The Kirkman graph $\mathfrak{K}$.\label{subfig:H13}]{
				\includegraphics[scale=1]{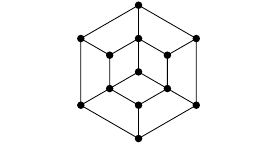}
			} &
			\subfloat[The graph $\mathfrak{H}_{15}$.\label{subfig:H15}]{
				\includegraphics[scale=1]{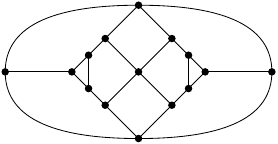}
			}
		\end{tabular}
		\caption{The Herschel family. In the drawing of \( \mathfrak{H}_n^\bullet \) (respectively, \( \mathfrak{H}_n^\circ \)), the rightmost vertical path degenerates into a single vertex when \( n = 11 \) (respectively, \( n = 13 \)).}
		\label{fig:Herschelfamily}
	\end{figure}
	
	The following result can be seen as an intermediate step toward Conjecture~\ref{con:K26}.
	
	\begin{theorem}[Ellingham, Gaslowitz, O'Connell, and Royle~\cite{O'Connell2018}]\label{thm:K115}
		Every non-hamiltonian polyhedral graph, other than the Herschel graph, contains a \( K_{1, 1, 5} \) minor.
	\end{theorem} 
	
	Despite these results, Question~\ref{que} is not yet fully resolved, since neither $K_{2,5}$ nor $K_{1,1,5}$ is 3-connected. To address this, Ding and Marshall~\cite{Ding2018} proposed a characterization of minimal minors for 3-connected non-hamiltonian graphs. Let $\mathfrak{Q}^+$ denote the graph obtained from the cube by adding a new vertex and connecting it to the three neighbors of a single original vertex.
	
	\begin{conjecture}[Ding and Marshall~\cite{Ding2018}] \label{con:DM}
		Every $3$-connected non-hamiltonian graph contains $K_{3, 4}$, $\mathfrak{Q}^+$, or $\mathfrak{H}$ as a minor.
	\end{conjecture}
	
	Since $K_{3, 4}$ and $\mathfrak{Q}^+$ are non-planar, Ding and Marshall's conjecture, if true, implies that every non-hamiltonian polyhedral graph must contain $\mathfrak{H}$ as a minor. They further proved the following supporting result.
	
	\begin{theorem}[Ding and Marshall~\cite{Ding2018}]\label{thm:Hp}
		Every non-hamiltonian plane triangulation contains $\mathfrak{H}$ as a minor.
	\end{theorem}
	
	Our main result establishes the planar case of Conjecture~\ref{con:DM}. That is, we show that the Herschel graph is the minimal non-hamiltonian polyhedral graph—not only with respect to order and size, but also with respect to the minor relation. This provides an optimal answer to Question~\ref{que}.
	
	\begin{theorem} \label{thm:H}
		Every non-hamiltonian polyhedral graph contains \(\mathfrak{H}\) as a minor.
	\end{theorem}
	
	Note that Theorem~\ref{thm:H} strengthens both Theorem~\ref{thm:Hp} and Theorem~\ref{thm:K25}, since $\mathfrak{H}$ contains $K_{2,5}$ as a minor.
	
	Let \( \mathfrak{H}^\bullet_n \), \( \mathfrak{H}^\circ_n \), \( \mathfrak{K} \), and \( \mathfrak{H}_{15} \) be the graphs shown in Figure~\ref{fig:Herschelfamily}, without any dashed edges. Note that $\mathfrak{H}^\bullet_{11} = \mathfrak{H}$. Define the \emph{Herschel family} to be the graph class
	\[
	\{\mathfrak{H}^\bullet_n : n \ge 11\} \cup \{\mathfrak{H}^\circ_n : n \ge 13\} \cup \{\mathfrak{K}, \mathfrak{H}_{15}\}.
	\]
	We apply Theorem~\ref{thm:H} to establish the following result.

	\begin{theorem} \label{thm:Hf}
		Every non-hamiltonian polyhedral graph without \( K_{2,6} \) minors contains a spanning subgraph from the Herschel family.
	\end{theorem}
	
	Finally, by applying Theorem~\ref{thm:Hf}, we are able to confirm Conjecture~\ref{con:K26}. 
	
	We also note that Theorem~\ref{thm:Hf} can provide a short derivation of Theorem~\ref{thm:K115}.

	While our results improve upon all previously mentioned works in this line of research—several of which involve computer assistance (see~\cite{Dillencourt1996, Ding2018, O'Connell2018})—our proofs, relying on a lemma by Ding and Marshall, adopt a considerably simpler approach and do not require computer assistance.
	
	In Section~\ref{sec:Herschel}, we prove Theorem~\ref{thm:H}. In Section~\ref{sec:K26}, we prove Theorem~\ref{thm:Hf} and establish Conjecture~\ref{con:K26}.

	\section{The minimal minor of non-hamiltonian polyhedra} \label{sec:Herschel}
	
	The main purpose of this section is to prove Theorem~\ref{thm:H}. Section~\ref{sec:pre} introduces the necessary definitions and develops some useful tools, while Section~\ref{sec:planar} presents the proof.
	
	\subsection{Preliminaries} \label{sec:pre}
	
Recall that a graph $H$ is a minor of a graph $G$ if it can be obtained from $G$ by a sequence of edge deletions, vertex deletions, and edge contractions. 
It is well known that $G$ contains $H$ as a minor if and only if there exists a mapping $\mu$ which maps $V(H)$ into pairwise disjoint subsets of $V(G)$ and maps $E(H)$ into pairwise internally disjoint paths in $G$, such that for every $v \in V(H)$, the induced subgraph $G[\mu(v)]$ is connected, and for any $uv \in E(H)$, the path $\mu(uv)$ joins $\mu(u)$ and $\mu(v)$ with no internal vertex of $\mu(uv)$ in $\bigcup_{w \in V(H)} \mu(w)$.
We say that such a mapping $\mu$ realizes an $H$ minor in $G$.
	
	Let \( G \) be a graph, and let \( k \geq 0 \) be an integer. A \emph{\( k \)-cut} is a subset \( S \) of vertices such that \( |S| = k \) and \( G - S \) is disconnected. A graph \( G \) is said to be \emph{\( k \)-connected} if \( |V(G)| > k \) and \( G \) has no \( k' \)-cut for any \( k' < k \). Furthermore, a planar graph \( G \) is \emph{internally $4$-connected} if it is 3-connected, has at least five vertices, and every 3-cut of \( G \) is an independent set of vertices whose removal disconnects \( G \) into a single vertex and another component.
	
	A graph \( G \) is a \emph{minor-minimal} counterexample to a statement if \( G \) is a counterexample, but no proper minor of \( G \) serves as a counterexample. Ding and Marshall~\cite{Ding2018} showed that any minor-minimal counterexample to their conjecture is internally 4-connected. Since every minor-minimal counterexample to Theorem~\ref{thm:H} also serves as a minor-minimal counterexample to the Ding–Marshall conjecture, we conclude the following.

	\begin{lemma}[\cite{Ding2018}] \label{lem:i4c}
		Every minor-minimal counterexample to Theorem~\ref{thm:H} is internally $4$-connected.
	\end{lemma}
	
	
	
	From the above lemma, we derive two further lemmas that will be frequently used in the proof.
	
	\begin{lemma} \label{lem:na4}
		Every minor-minimal counterexample to Theorem~\ref{thm:H} contains no pair of adjacent vertices, each of degree at least $4$.
	\end{lemma}
	\begin{proof}
		Suppose, to the contrary, that the minor-minimal counterexample $G$ has two adjacent vertices $v, w$ of degree at least 4. By the choice of $G$, $G - vw$ is no longer 3-connected. Let $\{x, y\}$ be a 2-cut of $G - vw$. Clearly, $G - vw - \{x, y\}$ has precisely two components, one of which contains $v$ and the other contains $w$. 
		Since $v$ has degree at least 4, $\{v, x, y\}$ is a 3-cut of $G$ that separates $w$ from a neighbor $u$ of $v$. 
		Let $A_w$ and $A_u$ be the components of $G - \{v, x, y\}$ containing $w$ and $u$, respectively. 
		Since $w$ has degree at least 4, $A_w$ contains some neighbor of $w$ and hence $|V(A_w)| > 1$. 
		It follows from Lemma~\ref{lem:i4c} that $\{v, x, y\}$ is an independent set and that $A_u$ consists solely of the vertex $u$.
		This implies that $v$ has a neighbor in $A_w$ other than $w$. Hence $G - vw - \{x, y\}$ is connected, which is a contradiction.
	\end{proof}
	
	\begin{lemma} \label{lem:n3c}
		Every minor-minimal counterexample to Theorem~\ref{thm:H} has no cycle of length $3$.
	\end{lemma}
	\begin{proof}
		Suppose, for contradiction, that the minor-minimal counterexample \( G \) contains a 3-cycle \( u v w u \). By Lemma~\ref{lem:na4}, we may assume that \( u \) has degree 3. Let \( u' \) be the neighbor of \( u \) other than \( v \) and \( w \). Then \( \{u', v, w\} \) forms a 3-cut, but it is not an independent set, contradicting Lemma~\ref{lem:i4c}.  
		%
		%
	\end{proof}

	Note that the proofs of Lemma~\ref{lem:na4} and Lemma~\ref{lem:n3c} do not rely on planarity and thus also apply to minor-minimal counterexamples to the conjecture by Ding and Marshall.

	\subsection{Proof of Theorem~\ref{thm:H}} \label{sec:planar}
	
	Throughout this section, we assume that $G$ is a minor-minimal counterexample to Theorem~\ref{thm:H}. We will derive a contradiction by showing that $G$ contains $\mathfrak{H}$ as a minor.
	
	Let $C$ be a cycle embedded in the plane, and let $u, v \in V(C)$. Define $[u, v]_C$ as the path in $C$ with end-vertices $u$ and $v$, such that it proceeds forward from $u$ to $v$ in a clockwise direction. With a slight abuse of notation, we may not distinguish between $[u, v]_C$ and $V([u, v]_C)$. We also define $(u, v)_C := [u, v]_C - \{u, v\}$. Note that $(u, v)_C$ may be empty. We omit the subscripts when it causes no ambiguity.
	
	Let \( O \) be a longest facial cycle of \( G \). We embed \( G \) such that \( O \) forms the boundary of the unbounded face. Define \( H := G - V(O) \). Since \( O \) is a facial cycle, it is an induced cycle in \( G \). 

	\begin{claim} \label{cla:2connected}
		$H$ is $2$-connected.
	\end{claim}
	\begin{proof}
		It is easy to show that \( H \) is connected.
		
		If there are two distinct vertices $v_1, v_2$ of $O$ that have a common neighbor $u$ in $H$, then, by Lemma~\ref{lem:n3c}, $v_1$ and $v_2$ are non-adjacent in $G$. Thus, $\{v_1, v_2, u\}$ is a 3-cut of $G$, which contradicts the fact in Lemma~\ref{lem:i4c} that every 3-cut of $G$ must be an independent set.
		Hence, any two distinct vertices of $O$ have no common neighbor in $H$, and $H$ has at least three vertices.
		
		Suppose $H$ has a cut-vertex $v$. Then there exist $u_1, u_2 \in V(O)$ such that $G - \{v, u_1, u_2\}$ has precisely two components, each of which contains at least one vertex of $H$. By Lemma~\ref{lem:i4c}, one of these components consists of a single vertex, implying that one of $(u_1, u_2)_O$ or $(u_2, u_1)_O$ is empty. Thus, $u_1$ and $u_2$ are adjacent in $O$. This means that $\{v, u_1, u_2\}$ is not an independent set, contradicting Lemma~\ref{lem:i4c}.	\end{proof}
	
	We consider the induced embedding of $H$ and let $C$ be the facial cycle of the unbounded face of $H$. Every vertex of $O$ has a neighbor in $C$ as $G$ is 3-connected and $O$ is an induced cycle. 
	
	Suppose that two vertices $v_1,v_2 \in V(O)$ share a common neighbor $u \in V(C)$. By planarity, either every vertex on $[v_1,v_2]_O$ or every vertex on $[v_2,v_1]_O$ is adjacent to $u$. Hence, some pair of adjacent vertices of $O$, together with $u$, induces a $3$-cycle in $G$, contradicting Lemma~\ref{lem:n3c}. Therefore, no two distinct vertices of $O$ share a common neighbor in $C$. We shall use this property hereafter without further mention.
	\begin{claim} \label{cla:nonemptyinterior}
		There exists some vertex in the interior of $C$.
	\end{claim}
	\begin{proof}
		Suppose to the contrary that the interior of $C$ contains no vertex. Then there exists some face in the interior of $C$ such that the corresponding facial cycle has at most one edge which is not in $C$. Note that, by Lemma~\ref{lem:n3c}, that facial cycle has length at least 4. Therefore, there exist two adjacent vertices $v_1, v_2$ each having degree 2 in $H$. By Lemma~\ref{lem:n3c}, $v_1$ and $v_2$ have no common neighbor in $G$. Thus, there exist $u_1, u_2 \in V(O)$ such that $v_1 v_2 u_2 u_1 v_1$ is a facial cycle of length 4. This implies that the union of $O - u_1u_2$, $C - v_1v_2$, $v_1u_1$, and $v_2u_2$ forms a Hamilton cycle of $G$, a contradiction.
	\end{proof}
	
	Let $A$ be a component of $H - V(C)$. Let $B$ be the union of $A$ and the edges joining $A$ to $C$ (and the end-vertices of these edges). The vertices of $B$ in $C$ are called the \emph{attachments} of $B$. As $G$ is 3-connected, $B$ has at least 3 attachments. Denote by $v_1, v_2, \dots, v_r$ the attachments of $B$ that occur in $C$ in this clockwise order. Write $v_{r+1} := v_1$. We call the paths $(v_1, v_2)_C$, $(v_2, v_3)_C$, \dots, $(v_r, v_{r+1})_C$ \emph{sectors} (with respect to $B$). 
	For any non-empty sector $(v_i, v_{i+1})$ ($i \in \{1, \dots, r\}$), $\{v_i, v_{i+1}\}$ is a 2-cut of $H$ that separates $(v_i, v_{i+1})$ from $(v_{i+1}, v_i)$. Therefore, since $G$ is 3-connected, any non-empty sector contains some vertex joining to $O$.
	
	The following two lemmas are useful for showing the existence of an \( \mathfrak{H} \) minor. The vertex labeling for \( \mathfrak{H} \) is given in Figure~\ref{subfig:Herschel}.

	\begin{lemma} \label{lem}
		Let $w_1, w_2, w_3$ be three attachments of $B$ that occur in this clockwise order in $C$. If $|V(O)| \ge 5$ and there exist $u_1, u_2, u_3 \in V(O)$ such that $u_1, u_2, u_3$ each have a neighbor in $(w_1, w_2)_C$, $(w_2, w_3)_C$, and $(w_3, w_1)_C$ respectively, then $G$ contains an $\mathfrak{H}$ minor.
	\end{lemma}
	\begin{proof}
		Let $u_4, u_5 \in V(O)$ be two vertices other than $u_1, u_2, u_3$. For $i \in \{1, \dots, 5\}$, let $x_i$ be a neighbor of $u_i$ in $C$ such that $x_1 \in (w_1, w_2)$, $x_2 \in (w_2, w_3)$, and $x_3 \in (w_3, w_1)$. Note that $x_1, \dots, x_5$ are distinct vertices. Without loss of generality, we only need to consider the case where $\{u_4, u_5\} \subseteq (u_1, u_2)_O$ and the case where $u_4 \in (u_1, u_2)_O$ and $u_5 \in (u_2, u_3)_O$.
		
		For the former case, we may further assume that $u_1, u_4, u_5, u_2$ occur in this clockwise order in $O$ and $\{x_4, x_5\} \not\subset [x_1, w_2]_C$. We have $x_5 \in (w_2, x_2)_C$. Set $\mu(h^1) := \{w_1\}$, $\mu(h^2) := \{u_5\}$, $\mu(h_1) := [x_2, w_3]_C$, $\mu(h_2) := (x_1, x_5)_C$, $\mu(h_3) := [u_3, u_1]_O$, $\mu(h_1^1) := V(A)$, $\mu(h_2^1) := \{x_1\}$, $\mu(h_3^1) := \{x_3\}$, $\mu(h_1^2) := \{x_5\}$, $\mu(h_2^2) := \{u_4\}$, and $\mu(h_3^2) := \{u_2\}$. It is clear that $\mu$ realizes an $\mathfrak{H}$ minor, as illustrated in Figure~\ref{subfig:lemma1}; one can easily find the required paths to form the whole minor, which we omit for ease of presentation.
		
		For the latter case, set $\mu(h^1) := \{w_1\}$, $\mu(h^2) := \{u_2\}$, $\mu(h_1) := (x_2, x_3)_C$, $\mu(h_2) := (x_1, x_2)_C$, $\mu(h_3) := [u_3, u_1]_O$, $\mu(h_1^1) := V(A)$, $\mu(h_2^1) := \{x_1\}$, $\mu(h_3^1) := \{x_3\}$, $\mu(h_1^2) := \{x_2\}$, $\mu(h_2^2) := \{u_4\}$, and $\mu(h_3^2) := \{u_5\}$. It is clear that $\mu$ realizes an $\mathfrak{H}$ minor, as illustrated in Figure~\ref{subfig:lemma2}.
	\end{proof}
	
	\begin{figure}[!ht]
		\centering
		\subfloat[$\{u_4, u_5\} \subseteq (u_1, u_2)_O$. \label{subfig:lemma1}]{%
			\includegraphics[scale=1]{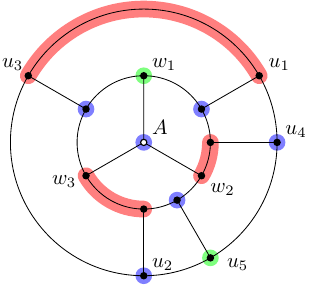}
		}
		\hspace{30pt}
		\subfloat[$u_4 \in (u_1, u_2)_O$ and $u_5 \in (u_2, u_3)_O$. \label{subfig:lemma2}]{%
			\includegraphics[scale=1]{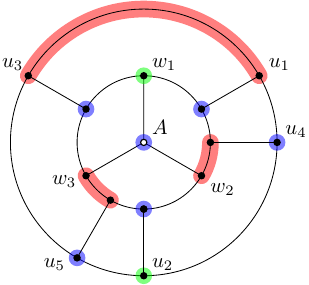}
		}
		\caption{Illustrations for the cases in the proof of Lemma~\ref{lem} demonstrating an $\mathfrak{H}$ minor in $G$.}
		
		\label{fig:lemma}
	\end{figure}
	
	\begin{lemma} \label{lew}
		Let $w_1,w_2,w_3$ be attachments of $B$ occurring on $C$ in clockwise order. Suppose there exist vertices $u,u_1,u_2,u_3 \in V(O)$ occurring on $O$ in clockwise order such that $u$ has a neighbor in $(w_1,w_2)_C$ and another neighbor in $(w_2,w_3)_C$, and $u_2$ has a neighbor in $(w_3,w_1)_C$. Then $G$ contains an $\mathfrak{H}$ minor.
	\end{lemma}
	
	\begin{proof}
		Let $z_1$ be a neighbor of $u$ in $(w_1, w_2)_C$ and $z_2$ be a neighbor of $u$ in $(w_2, w_3)_C$. Let $x_2$ be a neighbor of $u_2$ in $(w_3, w_1)_C$. So $u_1$ has a neighbor in $(z_2,x_2)_C$, and $u_3$ has a neighbor in $(x_2,z_1)_C$.
		
		Set $\mu(h^1) := \{w_2\}$, $\mu(h^2) := \{u_2\}$, $\mu(h_1) := \{u\}$, $\mu(h_2) := (z_2, x_2)_C$, $\mu(h_3) := (x_2, z_1)_C$, $\mu(h_1^1) := \{z_2\}$, $\mu(h_2^1) := V(A)$, $\mu(h_3^1) := \{z_1\}$, $\mu(h_1^2) := \{u_1\}$, $\mu(h_2^2) := \{x_2\}$, and $\mu(h_3^2) := \{u_3\}$. It is clear that $\mu$ realizes an $\mathfrak{H}$ minor in $G$.
	\end{proof}
	
	The following lemma gives a simple sufficient condition for $|V(O)| \ge 5$.
	
	\begin{lemma} \label{O5}
		If there is an empty sector, then $|V(O)| \ge 5$.
	\end{lemma}
	
	\begin{proof}
		As there is an empty sector, there are attachments $w_1,w_2$ of $B$ with $w_1 w_2 \in E(C)$. Since $w_1$ and $w_2$ are attachments of $B$, both have degree at least $3$ in $H$. By Lemma~\ref{lem:na4}, at least one of them has degree exactly $3$ in $G$. Without loss of generality, assume that $w_2$ has degree $3$ in $G$. So $w_2$ does not have any neighbor in $O$. Let $w_3$ be the neighbor of $w_2$ on $C$ distinct from $w_1$.
		
		By Lemma~\ref{lem:n3c}, every facial cycle of $G$ has length at least $4$. Suppose, for a contradiction, that the facial cycle containing the path $w_1w_2w_3$ has length exactly 4. Then it must be of the form $w_1w_2w_3vw_1$ for some vertex $v \in V(O)$. Since $w_1$ is an attachment of $B$, it has degree at least $4$ in $G$. Moreover, the vertex $v$ also has degree at least $4$ in $G$, as it has two neighbors in $O$ and two neighbors in $C$. This contradicts Lemma~\ref{lem:na4}. Therefore, the facial cycle of $G$ containing $w_1 w_2 w_3$ has length at least $5$. As $O$ is a longest facial cycle of $G$, it follows that $|V(O)| \ge 5$.
	\end{proof}

	We are now equipped to prove Theorem~\ref{thm:H}.

	Let $M$ be a matching such that every edge in $M$ has one end-vertex in $O$ and the other in $C$, with the end-vertices in $C$ belonging to distinct sectors (so there are at least $|M|$ non-empty sectors). Subject to these conditions, assume that $M$ is as large as possible.
	
	We now consider the following cases based on the value of $|M|$.
	
	\medskip
	\noindent \textbf{Case 1.} If $|M| = 0$, then all sectors are empty, and all vertices in $C$ are attachments of $B$. 
	
	It follows from Lemma~\ref{O5} that $|V(O)| \ge 5$.
	
	For any edge $uv \in E(O)$, there exist vertices $u',v' \in V(C)$ such that $u'uvv'$ is a subpath of a facial cycle of $G$. Since every vertex of $C$ is an attachment of $B$, both $u'$ and $v'$ have degree at least $4$. By Lemma~\ref{lem:na4}, the vertices $u'$ and $v'$ are non-adjacent. 
	
	Let $u_0u_1u_2u_3$ be a path on $O$ such that the vertices $u_0,u_1,u_2,u_3$ occur on $O$ in clockwise order. Applying the above observation to each of the edges $u_0u_1$, $u_1u_2$, and $u_2u_3$, we obtain, for each $i\in\{1,2,3\}$, vertices $u_{i-1}',x_i,w_i \in V(C)$ such that some facial cycle of $G$ contains $u_{i-1}'u_{i-1}u_ix_i$ as a subpath, with $w_i \in (u_{i-1}',x_i)_C$. Consequently, the vertices $x_1,x_2,x_3,w_1,w_2,w_3$ are pairwise distinct, and for each $i \in \{1,2,3\}$, we have $u_ix_i \in E(G)$ and $x_i \in (w_i,w_{i+1})_C$, where we denote $w_4:=w_1$. It now follows from Lemma~\ref{lem} that $G$ contains an $\mathfrak{H}$ minor, a contradiction.
	
	\medskip
	\noindent \textbf{Case 2.} Suppose $|M| = 1$. So there is at least one non-empty sector.
	
	If there is precisely one non-empty sector, then $B$ must have at least four attachments; otherwise, the three attachments would form a 3-cut of $G$ and, by Lemma~\ref{lem:i4c}, would be pairwise non-adjacent, yielding three non-empty sectors, a contradiction. Let $w_1, w_2$ be attachments of $B$ such that $(w_1, w_2)$ is the unique non-empty sector. All vertices in $[w_2, w_1]$ are attachments of $B$. Let $u_1$ be a vertex in $O$ that has a neighbor $x_1$ in $(w_1, w_2)$. 
	
	Since $B$ has at least four attachments, $(w_2, w_1)$ contains at least two vertices. By Lemma~\ref{lem:na4}, $u_1$ is not adjacent to any vertex in $(w_2, w_1)$. Furthermore, since $\{w_1, w_2\}$ is a 2-cut of $H$ and $G$ is 3-connected, there must be some vertex $u_2$ of $O$ that joins to $(w_2, w_1)$. In fact, there must exist at least one more vertex $u_3$ of $O$ that joins to $(w_2, w_1)$; otherwise, $\{w_1, w_2, u_2\}$ would form a 3-cut of $G$ whose removal results in one component containing $\{u_1, x_1\}$ and another containing $(w_2, w_1)$, contradicting Lemma~\ref{lem:i4c}. 
	
	We may assume that $u_1, u_2, u_3$ occur on $O$ in this clockwise order.
	Let $x_i$ be a neighbor of $u_i$ in $C$ for $i\in \{2,3\}$. As $x_2,x_3$ are attachments of $B$, they are not adjacent by Lemma~\ref{lem:na4}. This implies that $(x_2, x_3)$ contains an attachment of $B$. As there is only one non-empty sector, it follows from Lemma~\ref{O5} that $|V(O)|\ge5$. Therefore, by Lemma~\ref{lem}, $G$ contains an $\mathfrak{H}$ minor, a contradiction.
	
	If there are at least two non-empty sectors, then, since $|M|=1$ and $M$ is assumed to be as large as possible, there is a vertex $u$ in $O$ such that $u$ is the only vertex to which a non-empty sector can join. Specifically, there exist attachments $w_1, w_2, w_3$ of $B$ such that each of $(w_1, w_2)$ and $(w_2, w_3)$ contains at least one non-empty sector, while $(w_1, w_3)$ contains all non-empty sectors. Furthermore, $u$ is the only vertex in $O$ that joins to $(w_1, w_3)$, and every vertex in $O$ other than $u$ joins to $[w_3, w_1]$. Thus, $u$ has a neighbor in $(w_1, w_2)$ and a neighbor in $(w_2, w_3)$. Since $|V(O)| \ge 4$, there exist three distinct vertices in $O$ other than $u$, each of which joins to $[w_3, w_1]$. It follows from Lemma~\ref{lew} that $G$ contains an $\mathfrak{H}$ minor, a contradiction.
	
	\medskip
	\noindent \textbf{Case 3.} Suppose $|M| = 2$. So there are at least two non-empty sectors.
	
	Assume that there are precisely two non-empty sectors. Then there exist attachments $w_1,w_2,w_3,w_4$ of $B$ occurring on $C$ in clockwise order such that $w_1 \neq w_4$ (and possibly $w_2=w_3$), and $(w_1,w_2)_C$ and $(w_3,w_4)_C$ are the two non-empty sectors. By the case assumption, there exist vertices $u_1,u_2 \in V(O)$ such that $u_1$ joins to $(w_1,w_2)$ and $u_2$ joins to $(w_3,w_4)$. Moreover, as $(w_1,w_2)$ and $(w_3,w_4)$ are the only non-empty sectors, every vertex of $[w_2,w_3] \cup [w_4,w_1]$ is an attachment of $B$ and, by Lemma~\ref{O5}, $|V(O)|\ge5$.
	
	Suppose first that $w_2=w_3$ and $\{w_1,w_2,w_4\}$ is a $3$-cut of $G$. Since $G-\{w_1,w_4\}$ is connected, one component of $G-\{w_1,w_2,w_4\}$ contains $V(O)\cup (w_1,w_2) \cup (w_2,w_4)$, while another component contains $V(A)\cup (w_4,w_1)$, since every vertex of $(w_4,w_1)$ is an attachment of $B$. By Lemma~\ref{lem:i4c}, the latter component consists of a single vertex, and hence $(w_4,w_1)$ is empty. It follows that $w_4w_1 \in E(C)$, and thus $\{w_1,w_2,w_4\}$ is not an independent set, contradicting Lemma~\ref{lem:i4c}.
	
	Now suppose that $w_2=w_3$ and $\{w_1,w_2,w_4\}$ is not a $3$-cut of $G$. Then there exists a vertex $u \in V(O)$ adjacent to some attachment in $(w_4,w_1)$. By Lemma~\ref{lem:na4}, the vertices $u_1,u_2,u$ are distinct. It follows from Lemma~\ref{lem} that $G$ has an $\mathfrak{H}$ minor, a contradiction.
	
	Therefore, $w_2 \neq w_3$.
	
	By Lemma~\ref{lem}, we assert that there is no vertex in $O$ that joins to $(w_2, w_3)$ or $(w_4, w_1)$. Thus, as $|V(O)|\ge5$, it causes no loss of generality to assume that there are three vertices $u_3, u_4, u_5$ in $O$ each having a neighbor in $[w_1, w_2]$. Note that it is possible that $u_1$ is one of $u_3, u_4, u_5$. We assume $u_2, u_3, u_4, u_5$ occur in $O$ in this clockwise order. Let $D$ be the union of $B$, $[w_2, w_3]_C$ and $[w_4, w_1]_C$. We consider the induced embedding of $D$. Let $P_1$ be the path in $D$ with end-vertices $w_1, w_2$ contained in the boundary of the unbounded face such that $P_1$ is internally disjoint from $C$. Similarly, let $P_2$ be the path in $D$ with end-vertices $w_3, w_4$ contained in the boundary of the unbounded face such that $P_2$ is internally disjoint from $C$. It is clear that $P_1, P_2$ have length at least 2. 
	
	We show that $P_1$ and $P_2$ are disjoint. Suppose otherwise, and let $v$ be the common vertex of $P_1$ and $P_2$ that is closest to $w_1$ along $P_1$. Let $P_1'$ be the subpath of $P_1$ from $w_1$ to $v$, and let $P_2'$ be the subpath of $P_2$ from $w_4$ to $v$. By Lemma~\ref{lem:n3c}, $\{w_1, v, w_4\}$ does not induce a 3-cycle. Thus, at least one of $[w_4, w_1]$, $P_1'$, or $P_2'$ has length at least 2. Since no vertex in $O$ joins to $(w_4, w_1)$, it follows that $\{w_1, v, w_4\}$ is a 3-cut in $G$ separating $V(O)$ from the union of the interiors of $[w_4, w_1]$, $P_1'$, and $P_2'$. By Lemma~\ref{lem:i4c}, one component of $G - \{w_1, v, w_4\}$ consists of a single vertex $v'$, and $\{w_1, v, w_4\}$ must be an independent set. It follows that $v'$ is a common vertex of $P_1$ and $P_2$ that lies closer to $w_1$ along $P_1$ than $v$ does, contradicting the choice of $v$.
	
	Let $x_2$ be a neighbor of $u_2$ in $(w_3, w_4)_C$ and $x_i$ be a neighbor of $u_i$ in $[w_1, w_2]_C$ for every $i \in \{3, 4, 5\}$. Set $\mu(h^1) := \{u_4\}$, $\mu(h^2) := V(P_2) \setminus \{w_3, w_4\}$, $\mu(h_1) := [w_1, x_3]_C$, $\mu(h_2) := [x_5, w_2]_C$, $\mu(h_3) := \{u_2, x_2\}$, $\mu(h_1^1) := \{x_4\}$, $\mu(h_2^1) := \{u_5\}$, $\mu(h_3^1) := \{u_3\}$, $\mu(h_1^2) := V(P_1) \setminus \{w_1, w_2\}$, $\mu(h_2^2) := \{w_3\}$, and $\mu(h_3^2) := \{w_4\}$. It is clear that $\mu$ realizes an $\mathfrak{H}$ minor in $G$, as illustrated in Figure~\ref{fig:2}, which leads to a contradiction.
	
	Now, assume there are more than two non-empty sectors. Since $|M| = 2$, there exist not necessarily distinct attachments $w_1, w_2, w_3, w_4, w_5, w_6 \in V(C)$ occurring on $C$ in clockwise order and vertices $u_1, u_2 \in V(O)$ such that $(w_1, w_2)_C$, $(w_3, w_4)_C$, and $(w_5, w_6)_C$ are non-empty sectors, $u_1$ joins to both $(w_1, w_2)_C$ and $(w_3, w_4)_C$, and $u_2$ joins to $(w_5, w_6)_C$. It is clear that any vertex in $O$ other than $u_1$ and $u_2$ joins to $[w_4, w_1]$. Since $|V(O)| \ge 4$, it follows from Lemma~\ref{lew} that $G$ contains an $\mathfrak{H}$ minor, a contradiction.
	
	\begin{figure}[!ht]
		\centering
		\includegraphics[scale = 1]{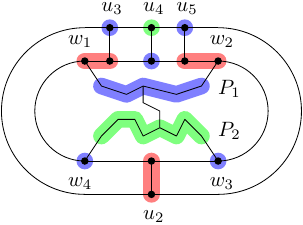}
		\caption{An $\mathfrak{H}$ minor in $G$ when $|M| = 2$.}
		\label{fig:2}
	\end{figure}
	
	\medskip
	\noindent \textbf{Case 4.} If $|M| \ge 3$, then there exist vertices $u_1, u_2, u_3$ in $O$ that occur in this clockwise order and attachments $w_1, w_2, w_3$ in $C$ that occur in this clockwise order such that for each $i \in \{1, 2, 3\}$, $u_i$ has a neighbor $x_i$ contained in $(w_i, w_{i + 1})$, where $w_4 := w_1$. 
	
	If $|V(O)| \ge 5$, then Lemma~\ref{lem} implies that $G$ contains an $\mathfrak{H}$ minor, a contradiction. Hence every facial cycle of $G$ has length $4$. Without loss of generality, we may assume that there is a vertex $u_4 \in V(O)$ such that $O=u_1u_2u_3u_4u_1$.
	
	We claim that $A$ is the only component of $H-V(C)$. Suppose, to the contrary, that there exists another component $A'$ of $H-V(C)$. Let $B'$ denote the union of $A'$ together with all edges joining $A'$ to $C$ and their end-vertices.
	Since $A$ was chosen arbitrarily among the components of $H-V(C)$, we may apply the preceding argument to $B'$ in place of $B$. This yields vertices $u'_1,u'_2,u'_3 \in V(O)$ occurring on $O$ in clockwise order and attachments $w'_1,w'_2,w'_3$ of $B'$ occurring on $C$ in clockwise order such that, for each $i \in \{1,2,3\}$, the vertex $u'_i$ has a neighbor $x'_i \in (w'_i,w'_{i+1})_C$, where we define $w'_4:=w'_1$. Without loss of generality, we may assume that $w_1,w_2,w_3,w'_1,w'_2,w'_3$ occur on $C$ in clockwise order.
	Set $\mu(h^1):=(x_1,x_2)_C$, $\mu(h^2):=(x'_1,x'_2)_C$, $\mu(h_1):=V(O)$, $\mu(h_2):=(x'_2,x_1)_C$, $\mu(h_3):=(x_2,x'_1)_C$, $\mu(h_1^1):=\{x_1\}$, $\mu(h_2^1):=V(A)$, $\mu(h_3^1):=\{x_2\}$, $\mu(h_1^2):=\{x'_2\}$, $\mu(h_2^2):=V(A')$, and $\mu(h_3^2):=\{x'_1\}$. It is straightforward to verify that $\mu$ realizes an $\mathfrak{H}$ minor in $G$, a contradiction.
	This proves that $H-V(C)=A$.

	Let $u_1u_2v_3v_4u_1$, where $v_3,v_4 \in V(G)$, denote the facial cycle of $G$ containing the edge $u_1u_2$ distinct from $u_1u_2u_3u_4u_1$. Clearly, $v_3v_4 \in E(C)$. By Lemma~\ref{lem:na4}, neither $u_1$ nor $u_2$ is adjacent to $w_2$. Hence $v_3v_4$ must be an edge of either $(w_1,w_2)_C$ or $(w_2,w_3)_C$.
	
	Suppose first that $v_3v_4$ is an edge of $(w_1,w_2)_C$. Then $u_2$ has degree at least $4$, and Lemma~\ref{lem:na4} implies that $v_3$ has degree $3$. Let $v_1v_2v_3v_4v_1$, where $v_1,v_2 \in V(G)$, denote the facial cycle of $G$ containing the edge $v_3v_4$ distinct from $u_1u_2v_3v_4u_1$. Since $v_3$ has degree $3$, we have $v_2 \in V(C)$.
	
	Suppose that $v_1 \notin V(C)$. Then $v_1 \in V(A)=V(H-V(C))$, and so $v_4$ is an attachment of $B$. Consequently, $u_2$ has a neighbor $v_3$ in $(v_4,w_2)_C$ and another neighbor $x_2$ in $(w_2,w_3)_C$. Since $u_3$ has a neighbor in $(w_3,w_1)_C$ and $u_1$ is adjacent to $v_4$, the vertex $u_4$ must have a neighbor in $(w_3,v_4)_C$. It now follows from Lemma~\ref{lew} that $G$ contains an $\mathfrak{H}$ minor, a contradiction.

Therefore, $v_1 \in V(C)$. By Lemma~\ref{lem:n3c}, the vertices $u_2$ and $v_2$ are non-adjacent. Hence $v_1$ must lie on $[w_1,v_4]_C$.

We claim that $v_1$ has degree $3$. Suppose otherwise. Then Lemma~\ref{lem:na4} implies that every neighbor of $v_1$ has degree $3$.
As every facial cycle containing $u_2$ and intersecting $[v_3,x_2]_C$ contains exactly three vertices of $[v_3,x_2]_C$, we may write $[v_3,x_2]_C=p_0q_1p_1\cdots q_tp_t$ for some positive integer $t$, where $p_0,q_1,p_1,\dots,q_t,p_t \in V(C)$ with $p_0=v_3$ and $p_t=x_2$. Since $u_2$ has degree at least $4$, Lemma~\ref{lem:na4} implies that each of $p_0,p_1,\dots,p_t$ has degree $3$. Consequently, the attachment $w_2$ belongs to $\{q_1,\dots,q_t\}$.
We next show that each of $q_1,\dots,q_t$ is adjacent to $v_1$. Since $q_1=v_2$, the vertex $v_1$ is adjacent to $q_1$. Suppose that $v_1$ is adjacent to $q_i$ for some $i \in \{1,\dots,t-1\}$. Then Lemma~\ref{lem:na4} implies that $q_i$ has degree $3$. Since $p_i$ also has degree $3$, the cycle $v_1q_ip_iq_{i+1}v_1$ is a facial cycle of $G$, and hence $v_1$ is adjacent to $q_{i+1}$. By induction, $v_1$ is adjacent to every vertex in $\{q_1,\dots,q_t\}$. In particular, $v_1$ is adjacent to $w_2$, contradicting Lemma~\ref{lem:na4}.

Therefore, $v_1$ has degree $3$. It follows that $(x_3,v_1)_C$ contains both $w_1$ and a neighbor of $u_4$. Consequently, one obtains an $\mathfrak{H}$ minor in $G$ by setting $\mu(h^1):=\{v_4\}$, $\mu(h^2):=\{w_3\}$, $\mu(h_1):=(x_3,v_1)_C$, $\mu(h_2):=(v_3,x_2)_C$, $\mu(h_3):=\{u_2\}$, $\mu(h_1^1):=\{v_1\}$, $\mu(h_2^1):=\{v_3\}$, $\mu(h_3^1):=\{u_1,u_4\}$, $\mu(h_1^2):=V(A)$, $\mu(h_2^2):=\{x_2\}$, and $\mu(h_3^2):=\{u_3,x_3\}$, a contradiction.

	This shows that $v_3 v_4$ is not an edge of $(w_1, w_2)_C$; hence, it must be an edge of $(w_2, w_3)_C$. Note that $x_1$ and $v_4$ are neighbors of $u_1$ in $(w_1, w_2)_C$ and $(w_2, w_3)_C$, respectively, and $x_3$ is a neighbor of $u_3$ in $(w_3, w_1)_C$. It follows from Lemma~\ref{lew} that $G$ contains an $\mathfrak{H}$ minor, a contradiction.

	\medskip
	
	This completes the proof of Theorem~\ref{thm:H}.
	
\section{Non-hamiltonian polyhedra without $K_{2,6}$ minors} \label{sec:K26}

In this section, we provide a characterization of non-hamiltonian polyhedral graphs that do not contain a $K_{2,6}$ minor. The main focus is the proof of Theorem~\ref{thm:Hf}, which is presented in Section~\ref{sec:Hfp}. This result sets the stage for addressing Conjecture~\ref{con:K26}, with its resolution provided in Section~\ref{sec:K26c}.

	\subsection{Proof of Theorem~\ref{thm:Hf}} \label{sec:Hfp}
	
This section is devoted to the proof of Theorem~\ref{thm:Hf}. Let $G$ be a non-hamiltonian polyhedral graph on $n$ vertices with no $K_{2,6}$ minor. We shall show that $G$ contains a spanning subgraph belonging to the Herschel family $\{\mathfrak{H}^\bullet_n : n \ge 11\} \cup \{\mathfrak{H}^\circ_n : n \ge 13\} \cup \{\mathfrak{K}, \mathfrak{H}_{15}\}$, where the graphs $\mathfrak{H}^\bullet_n$, $\mathfrak{H}^\circ_n$, $\mathfrak{K}$, and $\mathfrak{H}_{15}$ are depicted in Figure~\ref{fig:Herschelfamily} without the dashed edges.
	
	We emphasize that the symmetry of the Herschel graph plays a crucial role in our analysis. For clarity, we adopt the labeling shown in Figure~\ref{subfig:Herschel}. Recall that $\mathfrak{H}$ possesses twelve automorphisms, each either fixing $h^1$ or mapping it to $h^2$. When $h^1$ is fixed, there exists a unique automorphism corresponding to each permutation of $\{h_1^1, h_2^1, h_3^1\}$.
	
Throughout this section, subscripts in the labels of vertices of $\mathfrak{H}$ are taken modulo $3$.
	
	By Theorem~\ref{thm:H}, $G$ contains an $\mathfrak{H}$ minor.
	
	Let $H$ be a subgraph of $G$ such that $H$ contains an $\mathfrak{H}$ minor, whereas no proper subgraph of $H$ contains an $\mathfrak{H}$ minor. Then $H$ contains no $K_{2,6}$ minor. Let $\mu$ be a mapping realizing an $\mathfrak{H}$ minor in $H$ such that $\sum_{v \in V(\mathfrak{H})} |\mu(v)|$ is minimized. It is well known that if $v \in V(\mathfrak{H})$ has degree $3$ in $\mathfrak{H}$, then $|\mu(v)|=1$, while if $v \in V(\mathfrak{H})$ has degree $4$ in $\mathfrak{H}$, then $H[\mu(v)]$ is a path $S$ such that, whenever $|V(S)|>1$, each end-vertex of $S$ is incident with precisely two paths in $\{\mu(e): e \in E(\mathfrak{H})\}$. We call such a subgraph $H$ a \emph{Herschel frame}. A \emph{split} of $H$ is a vertex $v \in V(\mathfrak{H})$ satisfying $|\mu(v)|>1$.
	
	It follows immediately that $|\mu(v)|=1$ for every $v \in V(\mathfrak{H}) \setminus \{h_1,h_2,h_3\}$. Hence any split of $H$ must belong to $\{h_1,h_2,h_3\}$.
	
	For each $i \in \{1,2,3\}$, let $S_i$ denote the path $H[\mu(h_i)]$. Suppose that $h_i$ is a split of $H$ for some $i \in \{1,2,3\}$. Then, by planarity, either one end-vertex of $S_i$ is incident with $\mu(h_ih_{i-1}^1)$ and $\mu(h_ih_{i-1}^2)$ while the other end-vertex is incident with $\mu(h_ih_i^1)$ and $\mu(h_ih_i^2)$, or one end-vertex is incident with $\mu(h_ih_{i-1}^1)$ and $\mu(h_ih_i^1)$ while the other end-vertex is incident with $\mu(h_ih_{i-1}^2)$ and $\mu(h_ih_i^2)$. 
	The latter situation cannot occur, for otherwise $H$ would contain a $K_{2,6}$ minor, as illustrated in Figure~\ref{fig:HK26}. (Figure~\ref{fig:HK26} depicts the case $i=2$; the remaining cases follow by symmetry of the Herschel graph.) Moreover, $H$ has at most two splits, since otherwise it would contain a $K_{2,6}$ minor, as illustrated in Figure~\ref{fig:HK269}.
	
	\begin{figure}[!ht]
		\centering
		\includegraphics[scale = 1]{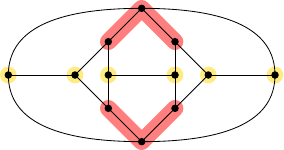}
		\caption{A graph containing a $K_{2,6}$ minor.}
		\label{fig:HK26}
	\end{figure}
	
	\begin{figure}[!ht]
		\centering
		\includegraphics[scale = 1]{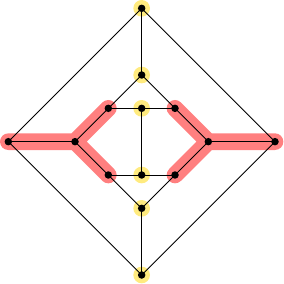}
		\caption{A graph containing a $K_{2,6}$ minor.}
		\label{fig:HK269}
	\end{figure}
	
	From now on, fix a Herschel frame $H$ such that $|V(H)|$ is minimized and, subject to this, the number of splits of $H$ is maximized.
	
	Recall that $\mu$ realizes an $\mathfrak{H}$ minor in $H$. We may also regard $\mu$ as realizing an $\mathfrak{H}$ minor in $G$. In particular, for each $i \in \{1,2,3\}$, the path $S_i$ is a spanning path of $G[\mu(h_i)]$. Whenever $|\mu(v)|=1$ for some $v \in V(\mathfrak{H})$, we identify the unique vertex in $\mu(v)$ with $v$. For any trail $T$ of $\mathfrak{H}$, we define $\mu(T)$ to be the union of the paths $\mu(e)$ over all edges $e \in E(T)$.
	
It immediately follows from the minimality of $|V(H)|$ that for any \( e \in E(\mathfrak{H}) \), \( \mu(e) \) is an induced path and, for any \( i \in \{1, 2, 3\} \), \( S_i \) is an induced path. This also implies that any internal vertex in any of these paths must have a neighbor outside that path (as $G$ is 3-connected).

	We divide the proof into the following cases, based on whether \( H \) is a spanning subgraph of \( G \) and the number of splits of $H$.

	\medskip
	\noindent \textbf{Case 1.} Suppose \( H \) is a spanning subgraph of \( G \) and has no split. 
	We aim to show that \( G \) contains either $\mathfrak{H}_n^\bullet$,  $\mathfrak{H}_n^\circ$, or $\mathfrak{H}_{15}$ as a spanning subgraph. We first prove the following claims.

	\begin{claim} \label{cla:cv}  
		If \( \mu(h_i h_i^j) \) or \( \mu(h_{i+1} h_i^j) \) contains an internal vertex \( v \) for some \( i \in \{1, 2, 3\} \) and \( j \in \{1, 2\} \), then \( h^j v \in E(G) \) and \( v \) has degree \( 3 \).  
	\end{claim}  
	\begin{proof}  
By symmetry of the Herschel graph, it suffices to consider the case in which $\mu(h_1h_1^1)$ contains an internal vertex $v$. Since $G$ is $3$-connected and planar, $H$ is a spanning subgraph of $G$, and $\mu(h_1h_1^1)$ is an induced path, it follows that $v$ must be adjacent to some vertex $u \in V\bigl(\mu(h_1h_1^2h_2h_1^1h^1h_3^1h_1)\bigr)\setminus \{h_1,h_1^1\}$.
		
		If $u \in V(\mu(h_1 h_1^2)) \setminus \{h_1\}$, it is straightforward to see that a Herschel frame with one more split can be obtained by modifying \( H \) as follows: replace \( S_1 \) with the subpath of \( \mu(h_1 h_1^1) \) joining \( h_1 \) and \( v \), replace \( \mu(h_1 h_1^1) \) with the subpath of \( \mu(h_1 h_1^1) \) joining \( h_1^1 \) and \( v \), and replace \( \mu(h_1 h_1^2) \) with the union of the edge \( vu \) and the subpath of \( \mu(h_1 h_1^2) \) joining \( h_1^2 \) and \( u \). However, this contradicts the maximality of the number of splits of \( H \).  
		
	If $u \in V(\mu(h_1^2 h_2 h_1^1 h^1 h_3^1 h_1)) \setminus \{h_1^2, h_1^1, h^1, h_1\}$, then \( G \) contains one of the graphs shown in Figure~\ref{fig:6} as a minor, and therefore contains a \( K_{2,6} \) minor, leading to a contradiction. 
		
		Thus, \( u \) must be \( h^1 \). The proof above also establishes that \( h^1 \) is the only neighbor of \( v \) not contained in \( \mu(h_1 h_1^1) \), which implies that \( v \) has degree \( 3 \).  
	\end{proof}

	\begin{figure}[!ht]
		\centering{%
			\includegraphics[scale=1]{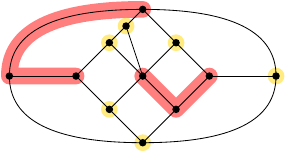}
		}
		\hfill
		{%
			\includegraphics[scale=1]{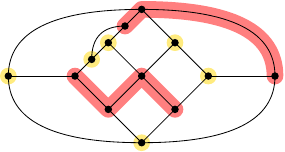}
		}
		\hfill
		{%
			\includegraphics[scale=1]{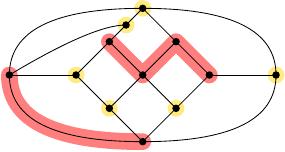}
		}
		\caption{Three graphs, each containing a \( K_{2,6} \) minor.}
		\label{fig:6}
	\end{figure}
	
	\begin{claim} \label{cla:oo}
		If \( \mu(h_i h_i^j) \) or \( \mu(h_{i+1} h_i^j) \) contains an internal vertex for some \( i \in \{1, 2, 3\} \) and \( j \in \{1, 2\} \), then any path \( \mu(h_k h_k^l) \) or \( \mu(h_{k+1} h_k^l) \) with \( k \in \{1, 2, 3\} \) and \( l \in \{1, 2\} \), other than \( \mu(h_i h_i^j) \) and \( \mu(h_{i+1} h_i^j) \), does not contain an internal vertex.
	\end{claim}
	\begin{proof}
		If the claim does not hold, then it follows from Claim~\ref{cla:cv} that $G$ contains one of the graphs depicted in Figure~\ref{fig:oo} as a minor, each of which contains a $K_{2,6}$ minor, which leads to a contradiction.
	\end{proof}
	
	\begin{figure}[!ht]
		\centering{%
			\includegraphics[scale=1]{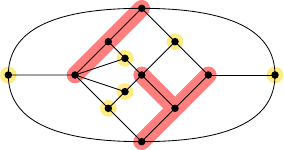}
		}
		\hfill
		{%
			\includegraphics[scale=1]{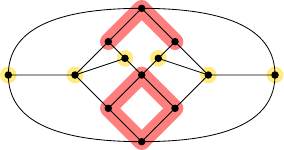}
		}
		\hfill
		{%
			\includegraphics[scale=1]{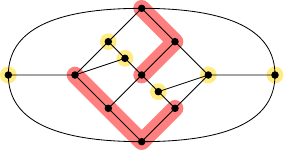}
		}
		\caption{Three graphs, each containing a \( K_{2,6} \) minor.}
		\label{fig:oo}
	\end{figure}
	
	\begin{claim} \label{cla:bl}
		If there is an edge joining an internal vertex of \( \mu(h^j h_i^j) \) and an internal vertex of \( \mu(h^j h_{i+1}^j) \) for some $i \in \{1, 2, 3\}$ and $j \in \{1, 2\}$, then each of \( \mu(h^j h_i^j) \) and \( \mu(h^j h_{i+1}^j) \) contains precisely one internal vertex, and \( \mu(h^j h_{i+2}^j) \) contains no internal vertex.
	\end{claim}
	\begin{proof}
		Without loss of generality, we assume \( i = j = 1 \). Let $v_1$ be the internal vertex of \( \mu(h^1 h_1^1) \) and $v_2$ the internal vertex of \( \mu(h^1 h_{2}^1) \) such that $v_1 v_2 \in E(G)$. If the subpath of \( \mu(h^1 h_1^1) \) joining $h^1$ and $v_1$ or the subpath of \( \mu(h^1 h_2^1) \) joining $h^1$ and $v_2$ contains an internal vertex $u$, then $G - u$ contains a Herschel frame, which contradicts the minimality of $|V(H)|$.
		
		By symmetry, it suffices to show that \( \mu(h^1 h_{3}^1) \) contains no internal vertex. Suppose \( \mu(h^1 h_{3}^1) \) has an internal vertex \( v \). Since \( G \) does not contain the graph in Figure~\ref{fig:b1} as a minor, \( v \) must be adjacent to some vertex in the subpath of \( \mu(h^1 h_1^1) \) joining \( v_1 \) and \( h_1^1 \), or to the subpath of \( \mu(h^1 h_2^1) \) joining \( v_2 \) and \( h_2^1 \). It can be readily shown that $G - h^1$, $G - v_1$, or $G - v_2$ contains a Herschel frame, a contradiction.
	\end{proof}

\begin{figure}[!ht]
	\centering
	\includegraphics[scale = 1]{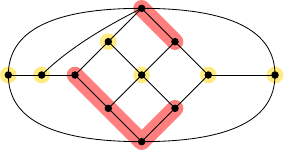}
	\caption{A graph containing a $K_{2,6}$ minor.}
	\label{fig:b1}
\end{figure}

	\begin{claim} \label{cla:iv}
		If \( \mu(h^j h_i^j) \) has an internal vertex $v$ for some $i \in \{1, 2, 3\}$ and $j \in \{1, 2\}$, then $v$ is adjacent to some vertex in \( V(\mu(h_{i+1}^j h^j h_{i+2}^j)) \setminus \{h^j\} \). Moreover, neither $\mu(h_k h_k^j)$ nor $\mu(h_k h_{k-1}^j)$ with $k \in \{1, 2, 3\}$ has an internal vertex, and at least one of \( \mu(h^j h_{i+1}^j) \) or \( \mu(h^j h_{i+2}^j) \) does not have an internal vertex.
	\end{claim}
	\begin{proof}
		Without loss of generality, we assume \( i = j = 1 \). 
		We can also assume that \( v \) is adjacent to some vertex \( u \) in \( V(\mu(h^1 h_2^1 h_2 h_1^1)) \setminus \{h^1, h_1^1\} \). In fact, \( u \in V(\mu(h^1 h_2^1)) \setminus \{h^1\} \), as otherwise, \( G \) would contain the graph in Figure~\ref{fig:b1} as a minor. It is straightforward to use the edge $uv$ to construct another Herschel frame, without changing the paths $\mu(h_k h_k^1)$ and $\mu(h_k h_{k-1}^1)$ for $k \in {1,2,3}$, and with $v$ replacing $h^1$. Therefore, by Claim~\ref{cla:cv}, if any of these six paths contains an internal vertex \( w \), then \( w \) must have degree 3 and be adjacent to both \( h^1 \) and \( v \), which is impossible.
		
		If \( u = h_2^1 \), then \( \mu(h^1 h_2^1) \) has no internal vertex, as there exists a Herschel frame that does not contain any internal vertices of \( \mu(h^1 h_2^1) \). If \( u \neq h_2^1 \), then, by Claim~\ref{cla:bl}, \( \mu(h^1 h_3^1) \) has no internal vertex. This justifies that at least one of \( \mu(h^1 h_2^1) \) or \( \mu(h^1 h_3^1) \) does not have an internal vertex.
	\end{proof}

We are now ready to apply the preceding claims to show that $G$ contains either $\mathfrak{H}_n^\bullet$, $\mathfrak{H}_n^\circ$, or $\mathfrak{H}_{15}$ as a spanning subgraph.

Assume first that for every $i \in \{1,2,3\}$ and $j \in \{1,2\}$, the path $\mu(h^j h_i^j)$ has no internal vertex. By Claim~\ref{cla:oo}, we may assume that if $\mu(e)$ has an internal vertex for some $e \in E(\mathfrak{H})$, then $e \in \{h_3h_3^2,h_1h_3^2\}$. By Claim~\ref{cla:cv}, every internal vertex of $\mu(h_3h_3^2)$ or $\mu(h_1h_3^2)$ is adjacent to $h^2$. Therefore, $G$ contains $\mathfrak{H}_n^\bullet$ as a spanning subgraph, where the path $\mu(h_3h_3^2h_1)$ corresponds to the path $h_3h_3^2 \dots h_{n-8}^2h_1$ shown in Figure~\ref{subfig:K26f1}.

Next, assume that for some $i \in \{1,2,3\}$ and $j \in \{1,2\}$ there exists an edge joining an internal vertex of $\mu(h^j h_i^j)$ and an internal vertex of $\mu(h^j h_{i+1}^j)$, and that for every $k \in \{1,2,3\}$ the path $\mu(h^{3-j} h_k^{3-j})$ has no internal vertex. Without loss of generality, assume $i=j=1$. By Claims~\ref{cla:oo},~\ref{cla:bl}, and~\ref{cla:iv}, we may assume that if $\mu(e)$ has an internal vertex for some $e \in E(\mathfrak{H})$, then $e \in \{h^1h_1^1,h^1h_2^1,h_3h_3^2,h_1h_3^2\}$, and each of $\mu(h^1h_1^1)$ and $\mu(h^1h_2^1)$ contains exactly one internal vertex. By Claim~\ref{cla:cv}, every internal vertex of $\mu(h_3h_3^2)$ or $\mu(h_1h_3^2)$ is adjacent to $h^2$. Therefore, $G$ contains $\mathfrak{H}_n^\circ$ as a spanning subgraph, where the path $\mu(h_3h_3^2h_1)$ corresponds to the path $h_3h_3^2 \dots h_{n-10}^2h_1$ shown in Figure~\ref{subfig:K26f2}.

Now assume that for some $i,k \in \{1,2,3\}$ there exists an edge joining an internal vertex of $\mu(h^1 h_i^1)$ and an internal vertex of $\mu(h^1 h_{i+1}^1)$, and another edge joining an internal vertex of $\mu(h^2 h_k^2)$ and an internal vertex of $\mu(h^2 h_{k+1}^2)$. Then Claims~\ref{cla:bl} and~\ref{cla:iv} imply that $G$ contains $\mathfrak{H}_{15}$ as a spanning subgraph.

Without loss of generality, it remains to consider the case where $\mu(h^1h_1^1)$ contains an internal vertex $v$ adjacent to both $h^1$ and $h_2^1$. Let $P$ be the union of $\mu(h_2^1h^1h_3^1)$, the edge $h_2^1v$, and $\mu(h^1h_1^1h_2)-h^1$. Then $P$ is a path joining $h_2$ and $h_3^1$.

Suppose first that for any $i \in \{1,2,3\}$ the path $\mu(h^2 h_i^2)$ has no internal vertex. By Claims~\ref{cla:iv} and~\ref{cla:oo}, for every $k \in \{1,2,3\}$ the path $\mu(h_kh_k^1h_{k+1})$ has length $2$, and at most one of $\mu(h_1h_1^2h_2)$, $\mu(h_2h_2^2h_3)$, and $\mu(h_3h_3^2h_1)$ has length greater than $2$. Consequently, the union of $P$ and $\mu(h_2h_1^2h_1h_3^2h^2h_2^2h_3h_3^1)$, the union of $P$ and $\mu(h_2h_2^2h_3h_3^2h^2h_1^2h_1h_3^1)$, or the union of $P$ and $\mu(h_2h_2^2h^2h_1^2h_1h_3^2h_3h_3^1)$ forms a Hamilton cycle, which is a contradiction.

Finally, assume that for some $i \in \{1,2,3\}$ the path $\mu(h^2 h_i^2)$ has an internal vertex. Then Claim~\ref{cla:iv} implies that for any $j\in\{1,2\}$ and any $k \in \{1,2,3\}$ the path $\mu(h_kh_k^jh_{k+1})$ has length $2$, and at least one of $\mu(h^2h_1^2)$, $\mu(h^2h_2^2)$, and $\mu(h^2h_3^2)$ has no internal vertex. Hence one of the three cycles described above is a Hamilton cycle of $G$, again yielding a contradiction.

	\medskip
	\noindent \textbf{Case 2.} Suppose \( H \) is a spanning subgraph of \( G \) with exactly one split. Without loss of generality, let \( h_2 \) be the split.
	
	We aim to show that \( G \) contains the Kirkman graph \( \mathfrak{K} \) as a spanning subgraph. To do so, we will first show that \( G \) contains a subdivision of \( \mathfrak{K} \) and then confirm that this subdivision has no internal vertices. We begin with the following claims.

	\begin{claim} \label{cla:bs}  
		For any \( j \in \{1, 2\} \), none of \( \mu(h_1 h_1^j) \), \( \mu(h_3 h_2^j) \), or \( \mu(h^j h_3^j) \) contains an internal vertex.  
	\end{claim}  
	\begin{proof}  
		By symmetry, it suffices to prove the claim for \( \mu(h_1 h_1^1) \) and \( \mu(h^1 h_3^1) \).  
		
		Suppose, for the sake of contradiction, that \( \mu(h_1 h_1^1) \) contains an internal vertex. Using reasoning similar to that in the proof of Claim~\ref{cla:cv}, one can readily deduce that either there exists a Herschel frame with one more split, or \( G \) contains the first or third graph in Figure~\ref{fig:6}, or the first graph in Figure~\ref{fig:H13} as a minor, which is a contradiction.  
		
		Now suppose \( \mu(h^1 h_3^1) \) contains an internal vertex. In this case, it is straightforward to verify that \( G \) contains either the graph in Figure~\ref{fig:b1} or the second graph in Figure~\ref{fig:H13} as a minor. Consequently, \( G \) would contain \( K_{2,6} \) as a minor, which again leads to a contradiction.  
	\end{proof}

	\begin{figure}[!ht]
		\centering{%
			\includegraphics[scale=1]{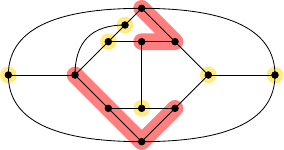}
		}
		\hspace{30pt}
		{%
			\includegraphics[scale=1]{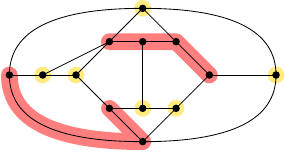}
		}
		\caption{Two graphs, each containing a \( K_{2,6} \) minor.}
		\label{fig:H13}
	\end{figure}
	
	\begin{claim} \label{cla:S2}
		The path $S_2$ contains some internal vertex.
	\end{claim}
	\begin{proof}
		
		For the purposes of this proof only, subject to the minimality of $|V(H)|$ and the maximality of the number of splits, we impose an additional assumption: the sum of the lengths of \( \mu(h^1 h_1^1) \) and \( \mu(h_1 h_3^1) \) is minimized while the path $S_2$ remains unchanged. We claim that neither \( \mu(h^1 h_1^1) \) nor \( \mu(h_1 h_3^1) \) contains an internal vertex.
		
		Suppose \( \mu(h^1 h_1^1) \) contains an internal vertex \( v \). Since \( G \) does not contain the graph in Figure~\ref{fig:b1} as a minor, \( v \) must be adjacent to a vertex in \( V(\mu(h_2^1 h^1 h_3^1)) \setminus \{h^1\} \). However, this would imply the existence of a Herschel frame satisfying the given conditions but with a smaller sum of the lengths of \( \mu(h^1 h_1^1) \) and \( \mu(h_1 h_3^1) \) while the path $S_2$ remains the same, which contradicts our assumption.
		
		Similarly, if \( \mu(h_1 h_3^1) \) contains an internal vertex, it follows that \( G \) either contains the first or third graph in Figure~\ref{fig:6} as a minor, or there exists a Herschel frame with one additional split, or with a smaller sum of the lengths of \( \mu(h^1 h_1^1) \) and \( \mu(h_1 h_3^1) \) while $S_2$ remains unchanged. All these scenarios contradict the assumptions. 
		
		Thus, we conclude that neither \( \mu(h^1 h_1^1) \) nor \( \mu(h_1 h_3^1) \) has an internal vertex.
		
		Observe that the path \( S_2 \) must contain an internal vertex; otherwise, by Claim~\ref{cla:bs} and the previous conclusion, 
$\mu(h^1 h_3^1 h_3 h_3^2 h_1 h_1^1 h_2 h_1^2 h^2 h_2^2 h_2 h_2^1 h^1)$ would be a Hamilton cycle of $G$.
	\end{proof}

\begin{claim}
    \label{cla:K}
    There is a subdivision of the Kirkman graph $\mathfrak{K}$ in $G$ obtained from the Herschel frame $H$ by adding an edge joining an internal vertex of $S_2$ to either $h^1$ or $h^2$.
\end{claim}

\begin{proof}
    For the purposes of this proof only, we impose an additional assumption: the length of \( S_2 \) is at least one and is minimized, which is consistent with the other assumptions. By Claim~\ref{cla:S2}, \( S_2 \) contains an internal vertex \( v \). As $G$ is 3-connected and $S_2$ is an induced path, $v$ has a neighbor in $\mu(h_2 h_1^1 h^1 h_2^1 h_2 h_2^2 h^2 h_1^2 h_2)-S_2$. Moreover, if that neighbor is neither $h^1$ nor $h^2$, then, by using the edge joining $v$ to that neighbor, one can obtain a Herschel frame in which the corresponding path $S_2$ is shorter while still having length at least one, contradicting the new assumption. It follows that $v$ must be adjacent to $h^1$ or $h^2$. This ensures that \( G \) contains a spanning subdivision of \( \mathfrak{K} \), as required.
\end{proof}

Thus, Claim~\ref{cla:K} guarantees the existence of a subdivision of $\mathfrak{K}$. It remains to show that this subdivision contains no internal vertices.

It is clear that the subdivision of $\mathfrak{K}$ contains three distinct Herschel frames, each of which has exactly one split and is a spanning subgraph of $G$. In particular, Claim~\ref{cla:bs} holds for each of these Herschel frames.

Without loss of generality, assume that the subdivision of $\mathfrak{K}$ is obtained from the Herschel frame by adding an edge joining an internal vertex \( v \) of $S_2$ to \( h^2 \).
	
	Let \( S_2' \) be the subpath of \( S_2 \) joining \( v \) and the common end-vertex of \( S_2 \) and \( \mu(h_1^1 h_2) \), and let \( S_2'' \) be the subpath of \( S_2 \) joining \( v \) and the common end-vertex of \( S_2 \) and \( \mu(h_2^1 h_2) \). 
	
	By the symmetry of \( \mathfrak{K} \) and by Claim~\ref{cla:bs}, it remains to verify that \( \mu(h^1 h_1^1) \) does not contain any internal vertex. Suppose it does, and let \( u \) be an internal vertex of \( \mu(h^1 h_1^1) \) adjacent to \( h^1 \). As in the previous arguments, \( u \) must be adjacent to some vertex not in \( \mu(h^1 h_1^1) \). It is not difficult to show that \( u \) must be adjacent to \( h_3^1 \), since otherwise \( G \) would contain either the graph in Figure~\ref{fig:b1} or the second graph in Figure~\ref{fig:H13} as a minor, leading to a contradiction. By symmetry, if \( S_2' \) has an internal vertex, this vertex must be adjacent to \( h_1^2 \). Therefore, we may assume that \( S_2' \) has no internal vertex (by adjusting the Herschel frame). 
	
	We now consider the union of 
	the path \( \mu(h_3^1 h^1 h_2^1 h_2) \), \( S_2'' \), the edge \( v h^2 \), the path \( \mu(h^2 h_2^2 h_3 h_3^2 h_1 h_1^2 h_2 h_1^1 h^1) - h^1 \), and the edge \( u h_3^1 \). This is a Hamilton cycle of \( G \), which leads to a contradiction. Therefore, \( G \) must contain a spanning subgraph isomorphic to \( \mathfrak{K} \), as claimed.
	
	\medskip
	\noindent \textbf{Case 3.} Suppose \( H \) is a spanning subgraph of \( G \) and has precisely two splits. Without loss of generality, let \( h_1 \) and \( h_3 \) be the splits. We will show that \( G \) contains a Hamilton cycle, thereby proving that this case cannot occur. We additionally assume that the sum of the lengths of the paths $\mu(h_2 h_i^j)$ with $i \in \{1, 2\}$ and $j \in \{1, 2\}$ is minimized.
	
	We will prove the following claims, showing that certain specific paths have no internal vertices, which aids in finding a Hamilton cycle.
	
	\begin{claim} \label{cla:bss}
		For any $i \in \{1, 2\}$ and any $j \in \{1, 2\}$, the path $\mu(h_2 h_i^j)$ does not contain any internal vertex. For any $j \in \{1, 2\}$, the path $\mu(h^j h_3^j)$ does not contain any internal vertex. None of $S_1$ or $S_3$ contains any internal vertex.
	\end{claim}
	\begin{proof}
		By the symmetry of the Herschel frame, it suffices to prove the claim for the paths \( \mu(h_2 h_1^1) \) and \( \mu(h^1 h_3^1) \).
		
		If \( \mu(h_2 h_1^1) \) contains an internal vertex, then there exists a Herschel frame such that either the sum of the lengths of the paths \( \mu(h_2 h_i^j) \) with \( i \in \{1, 2\} \) and \( j \in \{1, 2\} \) is smaller, or there is one additional split. Either possibility contradicts the maximality assumptions.
		
		If \( \mu(h^1 h_3^1) \) contains an internal vertex, then \( G \) would contain either the graph in Figure~\ref{fig:b1} or the first graph in Figure~\ref{fig:t} as a minor, both of which imply a \( K_{2,6} \) minor, which is impossible.
	\end{proof}

	\begin{figure}[!ht]
		\centering{%
			\includegraphics[scale=.75]{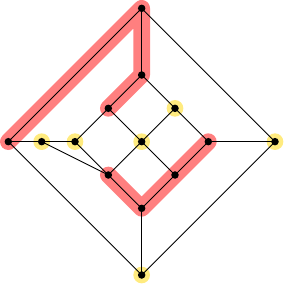}
		}
		\hfill
		{%
			\includegraphics[scale=.75]{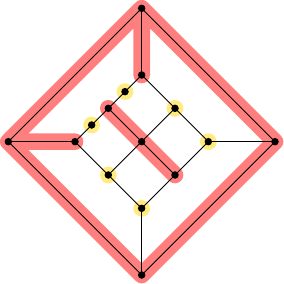}
		}
		\hfill
		{%
			\includegraphics[scale=.75]{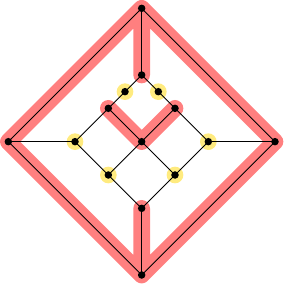}
		}
		\hfill
		{%
			\includegraphics[scale=.75]{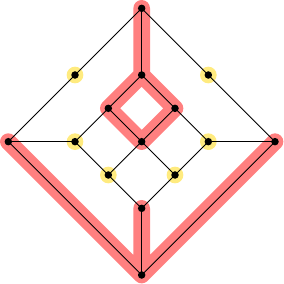}
		}
		\caption{Four graphs, each containing a \( K_{2,6} \) minor.}
		\label{fig:t}
	\end{figure}
	
	\begin{claim} \label{cla:tw}
		At least one of $\mu(h_1 h_1^1)$ or $\mu(h^1 h_1^1)$ does not have any internal vertex. At least one of $\mu(h_1 h_1^1)$ or $\mu(h_1 h_1^2)$ does not have any internal vertex. At least one of $\mu(h_1 h_3^1)$ or $\mu(h_1 h_3^2)$ does not have any internal vertex.
	\end{claim}
	\begin{proof}
		These three statements would lead to the second, third, and fourth graphs shown in Figure~\ref{fig:t}, respectively. However, this is not possible, as each of these graphs contains a \( K_{2,6} \) minor.
	\end{proof}

	We claim that, up to symmetry, it causes no loss of generality to assume that each of \( \mu(h_1 h_1^1) \), \( \mu(h^1 h_2^1) \), and \( \mu(h_1 h_3^1) \) has no internal vertices. To establish this, we repeatedly apply Claim~\ref{cla:tw}. 
	
	We first observe that at least one of the paths \( \mu(h_1 h_3^1) \), \( \mu(h_3 h_3^1) \), \( \mu(h_3 h_3^2) \), or \( \mu(h_1 h_3^2) \) must have no internal vertex; assume \( \mu(h_1 h_3^1) \) has no internal vertex. If \( \mu(h^1 h_1^1) \) contains an internal vertex, then we are done because both \( \mu(h_1 h_1^1) \) and \( \mu(h^1 h_2^1) \) would have no internal vertices. Therefore, assume \( \mu(h^1 h_1^1) \) has no internal vertex.
	
	If $\mu(h_1h_1^2)$ also has no internal vertex, then the claim follows (since $\mu(h^1h_1^1)$ and $\mu(h_1h_1^2)$ may play the roles of $\mu(h_1h_1^1)$ and $\mu(h^1h_2^1)$, respectively, while $\mu(h_1h_3^1)$ retains its original role). Otherwise, suppose \( \mu(h_1 h_1^2) \) contains an internal vertex. Then each of $\mu(h_1 h_1^1)$ and $\mu(h^2 h_1^2)$ has no internal vertex. If \( \mu(h_1 h_3^2) \) has no internal vertex, the claim is satisfied (since $\mu(h^2h_1^2)$, $\mu(h_1h_1^1)$, and $\mu(h_1h_3^2)$ may play the roles of $\mu(h_1h_1^1)$, $\mu(h^1h_2^1)$, and $\mu(h_1h_3^1)$, respectively). If \( \mu(h_1 h_3^2) \) does have an internal vertex, then \( \mu(h_3 h_3^2) \) must have no internal vertex. We can reason similarly for \( \mu(h_3 h_3^2) \) as we did for \( \mu(h_1 h_3^1) \). This shows that if the claim does not hold, \( \mu(h^2 h_1^2) \) must contain an internal vertex. However, this would not be possible, as \( \mu(h_1 h_1^2) \) contains some internal vertex.
	
	Thus, we conclude that we can assume each of the paths \( \mu(h_1 h_1^1) \), \( \mu(h^1 h_2^1) \), and \( \mu(h_1 h_3^1) \) has no internal vertices. It then follows from Claim~\ref{cla:bss} that the union of $S_1$ and $\mu(h_1 h_1^2 h^2 h_2^2 h_3 h_2^1 h_2 h_1^1 h^1 h_3^1 h_3 h_3^2 h_1)$ forms a Hamilton cycle in $G$, which is not possible.

	\medskip
	\noindent \textbf{Case 4.} Suppose \( H \) is not a spanning subgraph of \( G \). We will show that \( G \) is hamiltonian, leading to a contradiction.

	Let $A$ be a component of $G-V(H)$. Note that $A$ is contained in a face $F$ of $H$. Since $G$ is $3$-connected, there exist three edges joining $A$ to the boundary of $F$, and these edges have distinct end-vertices $w_1,w_2,w_3$ in $H$. By suitably contracting and removing edges in the union of $A$, $H$, and the joining edges, we may obtain one of the graphs in Figure~\ref{fig:u} as a minor of $G$, as follows.
	
	Without loss of generality, either $F$ is bounded by $\mu(h_1h_1^1h_2h_1^2h_1)$, or $F$ is bounded by the union of $\mu(h_2h_1^1h^1h_2^1h_2)$ and $S_2$. 
	
	We first consider the former case. If each of $\mu(h_1^1h_1h_1^2)$ and $\mu(h_1^1h_2h_1^2)$ contains one of $w_1,w_2,w_3$ as an internal vertex, then $G$ contains the second graph in Figure~\ref{fig:u} as a minor. Hence we may assume that $\mu(h_1^1h_1h_1^2)$ contains none of $w_1,w_2,w_3$ as an internal vertex. If all of $w_1,w_2,w_3$ are contained either in $\mu(h_1^1h_2)$ or in $\mu(h_1^2h_2)$, then $G$ contains the third graph in Figure~\ref{fig:u} as a minor. Otherwise, $G$ contains the first graph in Figure~\ref{fig:u} as a minor.
	
	We now consider the latter case. If each of $\mu(h^1h_1^1h_2)$ and $\mu(h^1h_2^1h_2)$ contains one of $w_1,w_2,w_3$ as an internal vertex, then $G$ contains the fourth graph in Figure~\ref{fig:u} as a minor. Thus we may assume that none of $w_1,w_2,w_3$ is an internal vertex of $\mu(h^1h_2^1h_2)$. Let $u$ denote the common end-vertex of $\mu(h_2h_1^1)$ and $\mu(h_2h_1^2)$. If each of $\mu(h^1h_1^1h_2)-u$ and $S_2-u$ contains one of $w_1,w_2,w_3$, or if $S_2$ contains all of $w_1,w_2,w_3$, then $G$ contains the fifth graph (which represents two graphs) or the sixth graph in Figure~\ref{fig:u} as a minor. Therefore, all of $w_1,w_2,w_3$ are contained in $\mu(h^1h_1^1h_2)$. Depending on whether all of $w_1,w_2,w_3$ are contained in $\mu(h^1h_1^1)$, all are contained in $\mu(h_1^1h_2)$, or neither holds, we conclude that $G$ contains the seventh, third, or eighth graph in Figure~\ref{fig:u}, respectively.
	
	\begin{figure}[!ht]
		\centering{%
			\includegraphics[scale=.75]{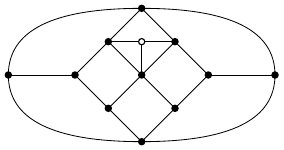}
		}
		\hfill
		{%
			\includegraphics[scale=.75]{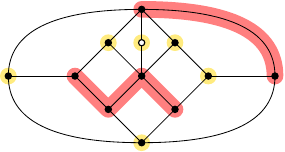}
		}
		\hfill
		{%
			\includegraphics[scale=.75]{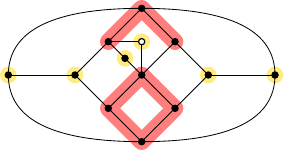}
		}
		\hfill
		{%
			\includegraphics[scale=.75]{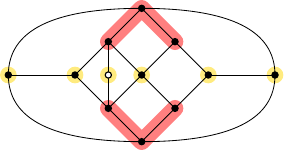}
		}
		
		\vspace{\floatsep}
		
		{%
			\includegraphics[scale=.75]{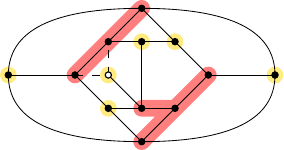}
		}
		\hfill
		{%
			\includegraphics[scale=.75]{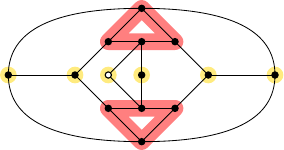}
		}
		\hfill
		{%
			\includegraphics[scale=.75]{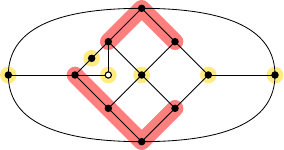}
		}
		\hfill
		{%
			\includegraphics[scale=.75]{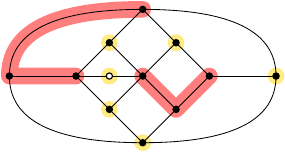}
		}
		\caption{Possible minors of the union of $A$, $H$, and the joining edges. The graph with two dashed edges represents two graphs, each containing exactly one dashed edge.}
		\label{fig:u}
	\end{figure}
	
	Therefore, $G$ must contain the first graph in Figure~\ref{fig:u} as a minor, since every other graph in Figure~\ref{fig:u} contains a $K_{2,6}$ minor. The preceding discussion also shows that $A$ is adjacent to exactly three vertices of $H$.

	We assume that \( A \) lies in the region bounded by \( \mu(h_1 h_1^1 h_2 h_1^2 h_1) \). In fact, there exists an edge connecting \( A \) to \( h_1^1 \) and another edge connecting \( A \) to \( h_1^2 \); otherwise, \( G \) would contain a \( K_{2,6} \) minor, as shown in the first graph of Figure~\ref{fig:v}. Moreover, the second graph in Figure~\ref{fig:v} shows that for any \( i, j \in \{1, 2\} \), the path \( \mu(h_i h_1^j) \) cannot contain any internal vertices. This implies that $A$ is adjacent to $h_1^1$, $h_1^2$, and $u$.

	\begin{figure}[!ht]
		\centering{%
			\includegraphics[scale=.75]{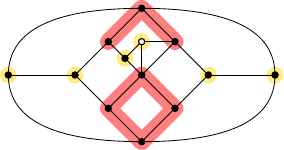}
		}
		\hfill
		{%
			\includegraphics[scale=.75]{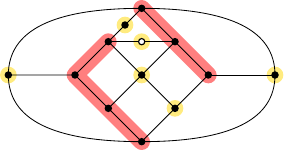}
		}
		\hfill
		{%
			\includegraphics[scale=.75]{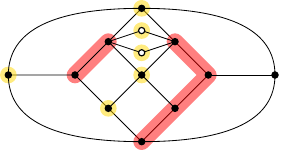}
		}
		\hfill
		{%
			\includegraphics[scale=.75]{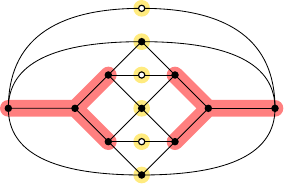}
		}
		\caption{Four graphs, each containing a $K_{2,6}$ minor.}
		\label{fig:v}
	\end{figure}
	
	The following claim aids in finding a path that spans a component of $G-V(H)$, which is instrumental in constructing a Hamilton cycle later.

	\begin{claim} \label{cla:A}
		Let $A$ be a component of $G-V(H)$ contained in the face bounded by $\mu(h_1h_1^1h_2h_1^2h_1)$, and suppose that $A$ is adjacent to $u$, where $u$ is the common end-vertex of $\mu(h_2h_1^1)$ and $\mu(h_2h_1^2)$. Let $A_0$ be the union of $A$, $\{h_1^1,h_1^2\}$, and the edges joining $A$ to $\{h_1^1,h_1^2\}$. Similarly, let $A_1$ be the union of $A$, $\{h_1^1,u\}$, and the edges joining $A$ to $\{h_1^1,u\}$. Then $A_0$ contains a spanning path joining $h_1^1$ and $h_1^2$, and $A_1$ contains a spanning path joining $h_1^1$ and $u$.
	\end{claim}
	
	\begin{proof}
		Observe that $A_0$ (respectively, $A_1$) is connected. Since $G$ is polyhedral and $A$ is adjacent only to $h_1^1,h_1^2,u$ in $H$, any cut-vertex of $A_0$ (respectively, $A_1$) must separate $h_1^1$ and $h_1^2$ (respectively, $h_1^1$ and $u$).
		
		Suppose that $A_0$ does not contain a spanning path with end-vertices $h_1^1$ and $h_1^2$. By \cite[Lemma~17]{Ellingham2016}, $A_0$ contains a $K_{2,2}$ minor rooted at $h_1^1$ and $h_1^2$. That is, there exist disjoint sets $U_1,U_2 \subset V(A_0)$ and distinct vertices $u_1,u_2 \in V(A_0)\setminus (U_1\cup U_2)$ such that, for each $j \in \{1,2\}$, the set $U_j$ contains $h_1^j$, induces a connected subgraph, and is adjacent to both $u_1$ and $u_2$. Consequently, $G$ contains the third graph in Figure~\ref{fig:v} as a minor, a contradiction.
		
		Similarly, if $A_1$ does not contain a spanning path with end-vertices $h_1^1$ and $u$, then \cite[Lemma~17]{Ellingham2016} implies that $A_1$ contains a $K_{2,2}$ minor rooted at $h_1^1$ and $u$. It follows that $G$ contains the third graph in Figure~\ref{fig:u} as a minor, again a contradiction.
	\end{proof}

	The third and fourth graphs in Figure~\ref{fig:v} illustrate that no two components of \( G - V(H) \) can reside in the same face of $H$, and \( G - V(H) \) has at most two components.

	Assume that \( G - V(H) \) consists of precisely one component \( A \), which lies within the face bounded by $\mu(h_1h_1^1h_2h_1^2h_1)$. Additionally, assume that the total length of the paths \( \mu(h^1 h_1^1) \), \( \mu(h_2 h_2^1) \), \( \mu(h_3 h_3^1) \), \( \mu(h_1 h_3^2) \), \( \mu(h^2 h_1^2) \), and \( \mu(h_3 h_2^2) \) is minimized. We claim that none of these six paths has internal vertices. 
	
	To establish this, we prove the claim for \( \mu(h^1 h_1^1) \) and omit the proofs for the remaining paths, as they follow by analogous arguments.

Suppose that $\mu(h^1h_1^1)$ contains an internal vertex $v$. Since $G$ does not contain the graph in Figure~\ref{fig:b1} as a minor, the vertex $v$ must have a neighbor in $V(\mu(h_2^1h^1h_3^1)) \setminus \{h^1\}$. By suitably modifying the Herschel frame, one can reduce the total length of the six paths without increasing $|V(H)|$ or decreasing the number of splits, contradicting our new assumption.

	Let \( P \) be the path with end-vertices \( h_1^1 \) and \( h_1^2 \) such that \( V(P) = \{h_1^1, h_1^2\} \cup V(A) \), which is assured by Claim~\ref{cla:A}. The union of \( P \), \( \mu(h_1 h_1^2) \), \( S_1 \), \( \mu(h_1 h_3^1 h^1 h_2^1 h_3) \), \( S_3 \), \( \mu(h_3 h_3^2 h^2 h_2^2 h_2) \), \( S_2 \), and \( \mu(h_2 h_1^1) \) forms a Hamilton cycle in \( G \), a contradiction.
	
	Finally, we assume $G- V(H)$ has precisely two components $A$ and $B$ and, without loss of generality, $A$ and $B$ lie within the faces bounded by $\mu(h_1h_1^1h_2h_1^2h_1)$ and by $\mu(h_2h_2^1h_3h_2^2h_2)$, respectively. It is easy to see that $H$ has no split.

	We further assume that the total length of the paths \( \mu(h^1 h_2^1) \), \( \mu(h_3 h_3^1) \), and \( \mu(h^2 h_3^2) \) is minimized. Similarly as before, one can show that these three paths do not have any internal vertices.
	
	Let \( P_A \) be the path with end-vertices \( h_1^1 \) and \( h_1^2 \) such that \( V(P_A) = \{h_1^1, h_1^2\} \cup V(A) \), as assured by Claim~\ref{cla:A}. Let $h_i$ with \( i \in \{2, 3\} \) be the vertex adjacent to \( B \). By Claim~\ref{cla:A}, there is a path \( P_B \) with end-vertices \( h_i \) and \( h_2^1 \) spanning \( \{h_i, h_2^1\} \cup V(B) \). Let \( P_B' \) be the union of \( P_B \) and \( \mu(h_{5-i} h_2^1) \).
	
	The union of \( P_A \), \( \mu(h_1^1 h^1 h_3^1 h_1 h_3^2 h_3) \), \( P_B' \), and \( \mu(h_2 h_2^2 h^2 h_1^2) \) forms a Hamilton cycle in \( G \), a contradiction.
	
	\medskip
	
	This completes the proof of Theorem~\ref{thm:Hf}.

	The final case of the above proof leads to the following consequence.
	
	\begin{corollary} \label{cor:Hf}
		Any non-hamiltonian polyhedral graph without $K_{2,6}$ minors does not contain a non-spanning Herschel frame.
	\end{corollary}

	\subsection{Proof of Conjecture~\ref{con:K26}} \label{sec:K26c}
	
	In this section, we present a proof of the conjecture by Ellingham et al., which asserts that every non-hamiltonian polyhedral graph with at least 16 vertices and no \( K_{2,6} \) minors is isomorphic to one of the graphs depicted in Figure~\ref{subfig:K26f1} or Figure~\ref{subfig:K26f2}, where each dashed edge may or may not be present.

	Let \( G_n^\bullet \) and \( G_n^\circ \) denote the graphs constructed from \( \mathfrak{H}_n^\bullet \) and \( \mathfrak{H}_n^\circ \), respectively, by adding all five dashed edges as depicted in Figure~\ref{fig:Herschelfamily}. The vertex labeling provided in Figure~\ref{fig:Herschelfamily} is naturally inherited.
	
	It is easy to see that for any $n \ge 11$, \( G_n^\bullet \) is hamiltonian if and only if \( G_{n+2}^\circ \) is hamiltonian. Note that \( G_n^\bullet \) does not have a Hamilton cycle because removing the vertices \( h_1 \), \( h_2 \), \( h_3 \), \( h^1 \), and \( h^2 \) disconnects the graph into six components.	Therefore, both \( G_n^\bullet \) and \( G_{n+2}^\circ \) are non-hamiltonian for all $n \ge 11$.
	
	Next, we show that neither \( G_n^\bullet \) nor \( G_n^\circ \) contains a \( K_{2,6} \) minor. For any \( n \geq 11 \), it is straightforward to see that \( G_n^\bullet \) is a minor of \( G_{n+2}^\circ \). Therefore, it suffices to prove the claim for \( G_n^\circ \) with \( n \geq 13 \). 
	
	Assume for contradiction that \( G_n^\circ \) contains a \( K_{2,6} \) minor for some \( n \). Take such an integer \( n \) to be as small as possible. Then, there exist disjoint vertex subsets \( A_1, A_2 \subseteq V(G_n^\circ) \) and distinct vertices \( b_1, \dots, b_6 \in V(G_n^\circ) \setminus (A_1 \cup A_2) \) such that \( A_i \) induces a connected subgraph, and  \( A_i \) is adjacent to \( b_j \) for all \( i \in \{1, 2\} \) and \( j \in \{1, \dots, 6\} \). Denote $B := \{b_1, \dots, b_6\}$. 
	
	Now consider the case where \( n = 13 \). To simplify the argument, we take advantage of the enhanced symmetry of the graph \( H \) obtained by adding the edge \( h_2h^2 \) to \( G_{13}^\circ \), and prove a stronger claim: \( H \) does not contain \( K_{2,6} \) as a minor.

	Without loss of generality, we assume \( |A_1| \leq |A_2| \), and thus \( |A_1| \leq 3 \). We then exhaust, up to the symmetry of \( H \), all configurations in which \( A_1 \) is adjacent to at least six vertices not contained in \( A_1 \). If \( |A_1| = 1 \), \( A_1 \) must be \( \{h_1\} \) or \( \{h^2\} \). If \( |A_1| = 2 \), \( A_1 \) must be \( \{h_1^1, h_1\} \), \( \{h_1, h_2\} \), \( \{h_1, h_1^2\} \), or \( \{h_1, h^2\} \). If \( |A_1| = 3 \), \( A_1 \) must be $\{t_1, h_1^1, h_1\}$, $\{t_1, h_1^1, h_2\}$, $\{h_1^1, h_1, h_2\}$, $\{h_1^1, h_1, h_3\}$, $\{h_1^1, h_1, h_1^2\}$, $\{h_1^1, h_1, h_3^2\}$, $\{h_1^1, h_1, h^2\}$, $\{h_1, h_2, h_3\}$, $\{h_1, h_2, h_1^2\}$, $\{h_1, h_2, h_2^2\}$, $\{h_1, h_2, h^2\}$, $\{h_1, h_1^1, h^2\}$, or $\{h_1, h_2^1, h^2\}$. A straightforward verification confirms that none of these configurations meet the conditions for a \( K_{2,6} \) minor. This establishes that this case cannot occur.
	
	Suppose \( n > 13 \). Let \( F := \{h_3^2, \dots, h_{n - 10}^2\} \). Clearly, \( A_i \not\subseteq F \), since \( A_i \) would then be adjacent to at most three vertices outside \( A_i \). Moreover, \( A_i \) cannot contain two consecutive vertices from \( F \), as this would lead to a \( K_{2,6} \) minor in \( G_{n-1}^\circ \). For the same reason, we conclude that \( F \subset A_1 \cup A_2 \cup B \).
	
	If \( h^2 \notin A_1 \cup A_2 \), then \( |B \cap F| \leq 1 \), and contracting \( F \) into a single vertex preserves the \( K_{2,6} \) minor, leading to a contradiction since \( G_{13}^\circ \) would then contain a \( K_{2,6} \) minor.
	
	Without loss of generality, assume \( h^2 \in A_1 \). Clearly, \( |B \cap F| \leq 2 \). 
	
	Suppose \( |B \cap F| = 2 \). Then \( A_2 \) must contain \( \{h_1, h_3\} \). If \( t_3 \in A_1 \), then \( h_2 \) must also be in \( A_1 \). However, in this case, \( A_2 \) can connect to at most two vertices from \( F \), at most one vertex from \( \{h_1^2, h_2^2\} \), at most one vertex from \( \{h_1^1, h_2^1, t_1, t_2\} \), and at most one vertex from \( \{h_3^1\} \). This contradicts that \( A_2 \) is adjacent to six vertices in \( V(G_{13}^\circ) \setminus (A_1 \cup A_2) \). If $A_1$ does not contain $t_3$, but contains $t_k$ for some \( k \in \{1, 2\} \), then \( \{h_{3-k}^1, t_{3-k}, t_3\} \subseteq B \). However, this is impossible because \( t_{3-k} \) cannot connect to \( A_2 \). If \( \{t_1, t_2, t_3\} \cap A_1 = \emptyset \), then \( A_1 \) can connect to at most three vertices from \( V(G_{13}^\circ) \setminus (A_1 \cup A_2 \cup F) \), again leading to a contradiction.
	
	Suppose \( |B \cap F| \le 1 \). If \( A_2 \cap F \neq \emptyset \), then, since \( A_2 \not\subseteq F \) and \( h^2 \in A_1 \), one can readily show that \( G_{n-1}^\circ \) also contains a \( K_{2,6} \) minor (by contracting an edge with one end-vertex in \( A_2 \cap F \) and the other end-vertex in \( A_2 \setminus F \)), which leads to a contradiction. Otherwise, we must have \( A_1 \cap F \neq \emptyset \). It is clear that if \( F \subset A_1 \), then \( G_{n-1}^\circ \) contains a \( K_{2,6} \) minor. Thus, we conclude that \( |B \cap F| = 1 \). By symmetry, we can assume \( h_1 \in A_2 \), \( h_{n-10}^2 \in B \), and \( F \setminus \{h_{n-10}^2\} \subseteq A_1 \) (since \( F \subseteq A_1 \cup A_2 \cup B \)). If \( h_3 \in A_2 \), we can apply the arguments from the previous paragraph to deduce a contradiction. Otherwise, we have three possibilities: \( h_3 \in A_1 \), \( h_3 \in B \), or \( h_3 \notin A_1 \cup A_2 \cup B \). In each case, one can again show that \( G_{n-1}^\circ \) contains a \( K_{2,6} \) minor.
	
	We have shown that both \( G_n^\bullet \) and \( G_n^\circ \) are non-hamiltonian and contain no \( K_{2,6} \) minors. 
	

	Now, let \( G \) be a non-hamiltonian polyhedral graph with \( n \geq 16 \) vertices and no \( K_{2,6} \) minors. By Theorem~\ref{thm:Hf}, \( G \) must contain \( \mathfrak{H}_n^\bullet \) or \( \mathfrak{H}_n^\circ \) as a spanning subgraph. To prove the conjecture, it remains to show that it is impossible for \( G \) to contain \( \mathfrak{H}_n^\bullet \) or \( \mathfrak{H}_n^\circ \) as a spanning subgraph while also having some additional edge other than the dashed edges \( h_1 h_2 \), \( h_2 h_3 \), \( h_3 h_1 \), \( h_1 h^2 \), and \( h_3 h^2 \).

	Suppose \( G \) contains \( \mathfrak{H}_n^\bullet \) as a spanning subgraph and has an additional edge \( e \) that is not one of the dashed edges. It is a straightforward fact that \( e \) lies within some face of \( \mathfrak{H}_n^\bullet \). 
	
If both end-vertices of $e$ belong to the set $\{h_1^1,h_2^1,h_3^1\} \cup \{h_1^2,h_2^2,h_3^2,h_{n-8}^2\}$ and $e \neq h_3^2h_{n-8}^2$, then $G$ contains a Hamilton cycle. To see this, observe that the graph obtained from $\mathfrak{H}_n^\bullet$ by adding $e$ is hamiltonian whenever the graph obtained by adding $e$ and contracting the path $h_3^2h_4^2\cdots h_{n-9}^2h_{n-8}^2$ is hamiltonian. Hence, by symmetry of the Herschel graph, it suffices to note that $h_1^1h^1h_3^1h_3h_2^2h^2h_3^2h_1h_1^2h_2h_2^1$ is a Hamilton path of $\mathfrak{H}$ joining $h_1^1$ and $h_2^1$, and that $h_1^1h^1h_3^1h_3h_2^1h_2h_2^2h^2h_3^2h_1h_1^2$ is a Hamilton path of $\mathfrak{H}$ joining $h_1^1$ and $h_1^2$.
	
If both end-vertices of \( e \) are from \( \{h_3^2, \dots, h_{n-8}^2\} \), then \( G \) contains a non-spanning Herschel frame, which, by Corollary~\ref{cor:Hf}, is not possible.
	
It remains to show that if $e=h^1h_i$ for some $i \in \{1,2,3\}$, or $e=h^2h_2$, or $e=h_3^1h_i^2$ for some $i \in \{4,\dots,n-9\}$, or $e=h_ih_k^2$ for some $i \in \{1,3\}$ and $k \in \{4,\dots,n-9\}$, then $G$ contains a $K_{2,6}$ minor. Indeed, in each case, $G$ contains one of the graphs depicted in Figure~\ref{fig:f} as a minor.

\begin{figure}[!ht]
	\centering{%
		\includegraphics[scale=1]{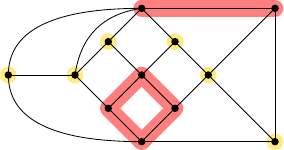}
	}
		\hfill
	{%
		\includegraphics[scale=1]{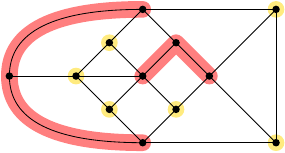}
	}
		\hfill
	{%
		\includegraphics[scale=1]{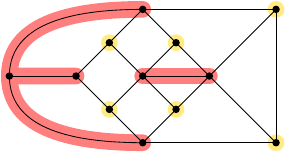}
	}
	
	\vspace{\floatsep}
	
	{%
		\includegraphics[scale=1]{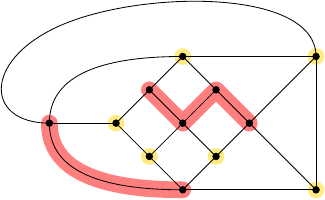}
	}
		\hspace{30pt}
	{%
		\includegraphics[scale=1]{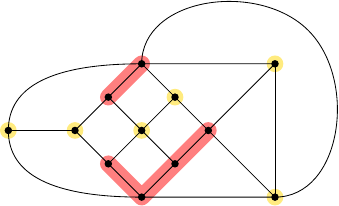}
	}
	\caption{Five graphs, each containing a $K_{2,6}$ minor.}
	\label{fig:f}
\end{figure}

	We omit the case where \( G \) contains \( \mathfrak{H}_n^\circ \) as a spanning subgraph and has an additional edge that is not one of the dashed edges, as the proof follows similarly.
	
	This completes the proof of Conjecture~\ref{con:K26}.

	We note that this characterization implies that for each \( n \ge 16 \), there are exactly 40 non-hamiltonian polyhedral graphs without \( K_{2,6} \) minors on \( n \) vertices. Moreover, as mentioned in \cite{O'Connell2018}, Gordon Royle’s computational results show that there are exactly 206 non-hamiltonian polyhedral graphs without \( K_{2,6} \) minors on fewer than 16 vertices.
	
	We remark that the Herschel family is, in fact, the class of \emph{edge-minimal} non-hamiltonian polyhedra that do not contain \( K_{2,6} \) minors.

	\section*{Acknowledgements}
	The authors thank the two anonymous referees for their comments, which have improved the presentation of this paper.
    The research of On-Hei Solomon Lo was supported by the Fundamental Research Funds for the Central Universities. The research of Kenta Ozeki was supported by JSPS KAKENHI Grant Numbers 22F22331, 22K19773, 23K03195, 26K00616, and 26K00618.
	
	\section*{Declaration of interests}
	
	The authors declare that they have no known competing financial interests or personal relationships that could have appeared to influence the work reported in this paper.

	
	\bibliographystyle{abbrv}
	\bibliography{paper}

\end{document}